\documentclass[]{amsart}

\usepackage[english]{babel}
\usepackage[T1]{fontenc}

\usepackage[a4paper]{geometry}
\usepackage{amsmath}
\usepackage{verbatim}
\usepackage[colorinlistoftodos]{todonotes}
\usepackage{mathtools} 

\usepackage[all]{xy}
\usepackage{amssymb,mathrsfs, amscd, graphicx, array, multirow,
enumerate, amsfonts, euscript}
\usepackage{comment}
\usepackage{upgreek} 

\usepackage{marvosym}
\xyoption{poly}
\usepackage{url}

\allowdisplaybreaks
\usepackage{color}
\usepackage{xcolor}
\def\redcolor#1{{\color{red} #1}}

\usepackage{tikz-cd}
\usepackage{tikz}
\usetikzlibrary{matrix,arrows}

\theoremstyle{plain}
\newtheorem{thm}{Theorem}[section]
\newtheorem{lemma}[thm]{Lemma}
\newtheorem{lem}[thm]{Lemma}
\newtheorem{prop}[thm]{Proposition}
\newtheorem{cor}[thm]{Corollary}
\newtheorem{conj}[thm]{Conjecture}

\theoremstyle{definition}
\newtheorem{question}[thm]{Question}
\newtheorem{que}[thm]{Question}
\newtheorem{defn}[thm]{Definition}

\theoremstyle{remark}
\newtheorem{remark}[thm]{Remark}
\newtheorem{rmk}[thm]{Remark}
\newtheorem{rem}[thm]{Remark}

\numberwithin{equation}{section}

\newcommand{\abs}[1]{\lvert #1 \rvert}

\def\sep{\mathrm{sep}}

\def\Art{{\rm Art}}

\def\makeop#1{\expandafter\def\csname#1\endcsname
  {\mathop{\rm #1}\nolimits}\ignorespaces}
\makeop{Hom}   \makeop{End}   \makeop{Aut}   \makeop{Isom}  \makeop{Pic} 
\makeop{Gal}   \makeop{ord}   \makeop{Char}  \makeop{Div}   \makeop{Lie} 
\makeop{PGL}   \makeop{Corr}  \makeop{PSL}   \makeop{sgn}   \makeop{Spf}
\makeop{Spec}  \makeop{Tr}    \makeop{Nr}    \makeop{Fr}    \makeop{disc}
\makeop{Proj}  \makeop{supp}  \makeop{ker}   \makeop{im}    \makeop{dom}
\makeop{coker} \makeop{Stab}  \makeop{SO}    \makeop{SL}    \makeop{SL}
\makeop{Cl}    \makeop{cond}  \makeop{Br}    \makeop{inv}   \makeop{rank}
\makeop{id}    \makeop{Fil}   \makeop{Frac}  \makeop{GL}    \makeop{SU}
\makeop{Nrd}   \makeop{Sp}    \makeop{Tr}    \makeop{Trd}   \makeop{diag}
\makeop{Res}   \makeop{ind}   \makeop{depth} \makeop{Tr}    \makeop{st}
\makeop{Ad}    \makeop{Int}   \makeop{tr}    \makeop{Sym}   \makeop{can}
\makeop{length}\makeop{SO}    \makeop{torsion} \makeop{GSp} \makeop{Ker}
\makeop{Adm}   \makeop{Mat}   \makeop{Cor}   \makeop{Inf} \makeop{Cok} \makeop{Coker}
\makeop{Ver}   \makeop{Ram}   \makeop{Br}
\def\makebb#1{\expandafter\def
  \csname bb#1\endcsname{{\mathbb{#1}}}\ignorespaces}
\def\makebf#1{\expandafter\def\csname bf#1\endcsname{{\bf
      #1}}\ignorespaces} 
\def\makegr#1{\expandafter\def
  \csname gr#1\endcsname{{\mathfrak{#1}}}\ignorespaces}
\def\makescr#1{\expandafter\def
  \csname scr#1\endcsname{{\EuScript{#1}}}\ignorespaces}
\def\makecal#1{\expandafter\def\csname cal#1\endcsname{{\mathcal
      #1}}\ignorespaces} 

\def\doLetters#1{#1A #1B #1C #1D #1E #1F #1G #1H #1I #1J #1K #1L #1M
                 #1N #1O #1P #1Q #1R #1S #1T #1U #1V #1W #1X #1Y #1Z}
\def\doletters#1{#1a #1b #1c #1d #1e #1f #1g #1h #1i #1j #1k #1l #1m
                 #1n #1o #1p #1q #1r #1s #1t #1u #1v #1w #1x #1y #1z}
\doLetters\makebb   \doLetters\makecal  \doLetters\makebf
\doLetters\makescr 
\doletters\makebf   \doLetters\makegr   \doletters\makegr
     \def\qed{\qedmark\medbreak}%
\def\qedmark{{\enspace\vrule height 6pt width 5pt depth 1.5pt}}%
    
\def\Gm{{{\bbG}_{\rm m}}}

\def\Spec{{\rm Spec}\,}

\def\Fp{{\bbF}_p}
\def\Fq{{\bbF}_q}

\def\Ql{{\bbQ}_{\ell}}

\def\Qp{{\bbQ}_p}

\def\Zp{{\bbZ}_p}
\def\Qbar{\overline{\bbQ}}

\def\wh{\widehat}
\def\wt{\widetilde}
\def\Ind{{\rm Ind}}

\def\Gm{\mathbb{G}_{\rm m}} 
\def\G{\mathbb{G}}
\def\R{\mathbb{R}}
\def\Q{\mathbb{Q}}
\def\Z{\mathbb{Z}}
\def\A{\mathbb{A}}
\def\C{\mathbb{C}}

\def\char{\text{char }}

\def\embed{\hookrightarrow}

\def\ol{\overline}

\def\pr{{\rm pr}}

\newcommand{\<}{\langle}   
\renewcommand{\>}{\rangle} 

\newcommand{\isoto}{\stackrel{\sim}{\to}}

\makeatletter
\newcommand{\xdashrightarrow}[2][]
  {\ext@arrow 0359\rightarrowfill@@{#1}{#2}}
\newcommand{\xdashleftarrow}[2][]
  {\ext@arrow 3095\leftarrowfill@@{#1}{#2}}
\newcommand{\xdashleftrightarrow}[2][]{\ext@arrow 3359\leftrightarrowfill@@{#1}{#2}}
\def\rightarrowfill@@{\arrowfill@@\relax\relbar\rightarrow}
\def\leftarrowfill@@{\arrowfill@@\leftarrow\relbar\relax}
\def\leftrightarrowfill@@{\arrowfill@@\leftarrow\relbar\rightarrow}
\def\arrowfill@@#1#2#3#4{%
  $\m@th\thickmuskip0mu\medmuskip\thickmuskip\thinmuskip\thickmuskip
   \relax#4#1
   \xleaders\hbox{$#4#2$}\hfill
   #3$%
}
\makeatother

\usepackage[OT2,T1]{fontenc}
\DeclareSymbolFont{cyrletters}{OT2}{wncyr}{m}{n}
\DeclareMathSymbol{\Sha}{\mathalpha}{cyrletters}{"58}

\def\Gmk{\G_{{\rm m}, k}}

\def\GmK{\G_{{\rm m}, K}}

\begin{document}

\title[Tamagawa numbers of CM tori]{On Tamagawa numbers of CM tori}

\author{Pei-Xin Liang}
\address{(Liang) Department of Mathematics, National Tsing Hua University
Taipei, Taiwan, }
\email{cindy11420@gmail.com}

\author{Yasuhiro Oki}
\address{(Oki) Department of Mathematics, Faculty of Science, Hokkaido University, 060-0810, Sapporo, Hokkaido, Japan}
\email{oki@math.sci.hokudai.ac.jp}

\author{Hsin-Yi Yang}
\address{(Yang) 
Fakult\"at f\"ur Mathematik,
Universit\"at Duisburg-Essen,
45117 Essen, Germany
}
\email{hsin-yi.yang@stud.uni-due.de}


\author{Chia-Fu Yu}
\address{(Yu) Institute of Mathematics, Academia Sinica and NCTS, Taipei, Taiwan, 10617}
\email{chiafu@math.sinica.edu.tw}

\date{\today}
\subjclass[2010]{14K22, 11R29.} 
\keywords{CM algebraic tori, Tamagawa numbers.}  

\maketitle

\begin{abstract}
In this article we investigate the problem of computing Tamagawa numbers of CM tori. This problem arises naturally from the problem of counting polarized abelian varieties with commutative endomorphism algebras over finite fields, and polarized CM abelian varieties and components of unitary Shimura varieties in the works of Achter--Altug--Garcia--Gordon and of Guo--Sheu--Yu, respectively. We make a systematic study on Galois cohomology groups in a more general setting and compute the Tamagawa numbers of CM tori associated to various Galois CM fields. Furthermore, we show that every (positive or negative) power of $2$ is the Tamagawa number of a CM tori, proving the analogous conjecture of Ono for CM tori.     
\end{abstract}

\section{Introduction}
\label{sec:I}

In his two fundamental papers \cite{ono:arithmetic, ono:tamagawa} Takashi Ono investigated the arithmetic of algebraic tori. He introduced and explored the class number and Tamagawa number of $T$, which will be denoted by $h(T)$ and $\tau(T)$ respectively (also see Section 2 for the definitions). One arithmetic significant of these invariants is that
the class number $h(\G_{{\rm m},k})$ is equal to the class number $h_k$ of the number field $k$, and the analytic class number formula for $k$ can be reformulated by the simple statement 
$\tau(\G_{{\rm m},k})=1$. Thus,  the class numbers of algebraic tori can be viewed as generalizations of class numbers of number fields, while Tamagawa numbers play a key role in the extension of analytic class number formulas.   

Ono showed \cite{ono:tamagawa} that $\tau(T)=|H^1(k, X(T))|/|\Sha^1(k,T)|$, where $X(T)$ is the group of characters of $T$ and $\Sha^1(k,T)$ is the Tate-Shafarevich group of $T$. Kottwitz \cite{kottwitz:duke1984} generalized Ono's formula to reductive groups and proved \cite{Kottwitz-Tamagawa-numbers} the celebrated conjecture of Weil for the Tamagawa number of semi-simple simply connected groups. Ono constructed a $15$-dimensional algebraic torus with Tamagawa number $1/4$, showing that $\tau(T)$ can be non-integral and conjectured in \cite{ono:sugaku} that every positive rational number is equal to $\tau(T)$ for some torus $T$. Ono's conjecture was proved by S.~Katayama \cite{katayama:1985} for the number field case. 
For some later studies of class numbers and Tamagawa numbers of algebraic tori we refer to J.-M.~Shyr \cite[Theorem 1] {Shyr-class-number-relation}, 
S. Katayama \cite{katayama:kyoto1991}, M.~Morishita \cite{morishita:nagoya1991}, C.~Gonz\'{a}lez-Avil\'{e}s \cite{gonzalez:mrl2008,gonzalez:crelle2010} and M.-H.~Tran \cite{tran:jnt2017} and references within.

In this article we are mainly concerned with the problem of computing the Tamagawa numbers of complex multiplication (CM) algebraic tori. CM tori are closely related to the arithmetic of CM abelian varieties and computing their Tamagawa numbers itself is a way of exploring the structure of CM fields. This problem directly contributes to recent works of Achter, Altug, Garcia and Gordon \cite{achter-altug-gordon} and of J.~Guo, N.~Sheu and the fourth named author \cite{guo-sheu-yu:CM}. 
In the former one the authors computed the size of an isogeny class of principally polarized abelian varieties over a finite field with commutative endomorphism algebra, and express the number in terms of a discriminant, the Tamagawa number, and the product of Frobenius local densities. In the latter one the authors computed formulas for certain CM abelian varieties and certain polarized abelian varieties over finite fields with commutative endomorphism algebras upon the results of \cite{xue-yu:counting}. Using the class number formula for CM tori, they also computed the numbers of connected components of complex unitary Shimura varieties.
In the appendix of \cite{achter-altug-gordon}, W.-W. Li and T.~R\"ud have obtained several initial results of the values of $\tau(T)$. Our goals are to prove more cases of CM tori and to determine the range of the values of Tamagawa numbers of all CM tori.
With a similar goal but using different methods, T.~R\"ud obtains several results along this direction \cite{rued:thesis}. He provides an algorithm, among others, for giving precise lower bounds and determining possible Tamagawa numbers, and obtains the values $\tau(T)$ for several other CM tori of lower dimension. \\

We shall describe our results towards computing Tamagawa numbers for a more general class of algebraic tori which include CM tori and then give more detailed results of CM tori. Let $k$ be a global field and $K:=\prod_{i=1}^rK_{i}$ be the product of finite separable field extensions $K_i$ of $k$. Let $E:=\prod_{i=1}^r E_i$, where each $E_i\subset K_i$ is a subextension of $K_i$. Denote by $T^K=\prod_i T^{K_i}$ and $T^E=\prod_i T^{E_i}$ the algebraic $k$-tori associated to the multiplicative groups of $K$ and $E$, respectively, and let $N_{K/E}=\prod_i N_{K_i/E_i}: T^K\to T^E$ be the norm map. We write $T^{K/E,1}$ for the kernel of $N_{K/E}$ and $T^{K/E,k}:=N_{K/E}^{-1}(\Gmk)$ for the preimage of the subtorus $\Gmk \embed T^E$ via the diagonal embedding. Let $L$ be the smallest splitting field of $T^{K/E,k}$ and let $G=\Gal(L/k)$. We let $\Lambda:=X(T^{K/E,k})$ and $\Lambda^1:=X(T^{K/E,1})$ be the character groups of $T^{K/E,k}$ and $T^{K/E,1}$, respectively. They fit in the following short exact sequence of $G$-lattices:
\begin{equation}
  \label{eq:I.1}
  \begin{CD}
    0 @>>> \Z @>>> \Lambda @>>> \Lambda^1 @>>> 0. 
  \end{CD}
\end{equation}
Write $H_i:=\Gal(L/K_i)$,  $\wt N_i:=\Gal(L/E_i)$, and $N_i^{\rm ab}:=\wt N_i/D(\wt N_i)H_i$, where $D(\wt N_i)$ denotes the commutator group of $\wt N_i$.
When there is no confusion, 
for brevity we shall write $H^q(A)$ for $H^q(G,A)$ for any $G$-module $A$. 

Let $\Ver_{G,N_i}:G\to N_i^{\rm ab}$ denote the transfer from $G$ into $N_i^{\rm ab}$; see Definition~\ref{def:rt}.
For any abelian group $H$, the Pontryagin dual of $H$ is denoted by $H^\vee$. By Ono's formula (cf.~\eqref{eq:ono-1}) and
the Poitou-Tate duality, 
we have
$\tau(T)=|H^1(G,X(T))|/|\Sha^2(G,X(T))|$, where $\Sha^{i}(G,X(T))$ is the $i$th Tate-Shafarevich group of $X( T)$. 

Let $\scrD$ be the set of all decomposition groups of $G$.    
Denote by
\[ r_{\scrD,A}^i: H^i(G,A) \to \bigoplus_{D\in \scrD} H^i(D,A)\] 
the restriction map for each $G$-module $A$. Put
\begin{equation}
  \label{eq:I.1.5}
  H^2(\Z)':=\{x\in H^2(G, \Z) : r^2_{\scrD,\Z} (x) \in {\rm Im} (\delta_\scrD) \},  
\end{equation}
where $\delta_\scrD: \bigoplus_{D\in \scrD} H^1(D, \Lambda^1) \to    \bigoplus_{D\in \scrD} H^2(D, \Z)$ is the connecting homomorphism induced from \eqref{eq:I.1}. 
The group $H^2(\Z)'$ plays a similar role of a Selmer group.  

\begin{thm}\label{I.1}
  Let the notation be as above.

{\rm (1)} There is a canonical isomorphism $H^1(G,\Lambda^1)\simeq \bigoplus_{i} N_{i}^{\rm ab,\vee}$.  

{\rm (2)} There is a canonical isomorphism
\begin{equation}
    \label{eq:I.2}
    H^1(G,\Lambda)\simeq \Ker\left (\sum \Ver_{G,N_i}^{\vee}: \bigoplus_i N_i^{\rm ab,\vee} \to G^{\rm ab,\vee} \right ),
\end{equation}
where $\Ver_{G,N_i}:G\to N_i^{\rm ab}$ is the transfer map.

{\rm (3)}  Assume that $K_i/E_i$ is cyclic with Galois group $N_i$ for all $i$. Then $\Sha^2(\Lambda)\simeq H^2(\Z)'/{\rm Im}(\delta)$ and
  \begin{equation}
    \label{eq:I.5}
\tau(T^{K/E,k})=\frac{\prod_{i=1}^r |N_i|}{|H^2(\Z)'|}.    
  \end{equation}
  
{\rm (4)} If we further assume that the subgroups $\wt N_i$ and $H_i$ are all normal in $G$, and   
  let
  \[ \Ver_{G, \bfN}=(\Ver_{G, N_i})_i : G^{\rm ab} \to \prod_i N_i \text{ and }
     \Ver_{D, \ol \bfD}=(\Ver_{D, \ol D_i})_i: D^{\rm ab}\to \prod_{i} \ol D_i \]  
   denote the corresponding transfer maps, where $\ol D_i:=D_i/(D_i \cap H_i)\subset N_i$, respectively, then
\begin{equation}
  \label{eq:I.6}
  H^2(\Z)'=\{f \in G^{{\rm ab}, \vee}: f {\big |}_{D^{\rm ab}}
  \in {\rm Im}(\Ver_{D, \ol \bfD}^\vee)\  \forall\, D\in \scrD \}, \quad \text{and}
\end{equation}
\begin{equation}
  \label{eq:I.7} 
  \Sha^2(\Lambda) \simeq \frac{\{f \in G^{{\rm ab}, \vee}: f {\big |}_{D^{\rm ab}} \in {\rm Im}(\Ver_{D, \ol \bfD}^\vee) \ \forall\, D\in \scrD \}}{{\rm Im} (\Ver^{\vee}_{G,\bfN})}.
\end{equation}  
\end{thm}

Using Theorem~\ref{I.1}(2), we give a different proof of a result of Li and R\"ud \cite[Proposition A.11]{achter-altug-gordon} which does not rely on Kottwitz's formula; see Corollary~\ref{cor:A.11}. 
By \eqref{eq:I.5}, the ratio $(\prod_{i=1}^r |N_i|) / \tau(T^{K/E,k})$ is a positive integer; this gives a simple upper bound of $\tau(T^{K/E,k})$.


\begin{thm}\label{I.2} Suppose the following conditions hold: 
(a) for each $1\le i \le r$, the extension $K_i/k$ is Galois with Galois group $G_i$, $N_i=\Gal(K_i/E_i)$ is cyclic and every decomposition group of $G_i$ is cyclic; and (b) $G\simeq G_1\times \dots \times G_r$. Then we have $\tau(T^{K/E,k})=\prod_{i=1}^r \tau(T^{K_i/E_i,k})$.
\end{thm}

Now let $K=\prod_{i=1}^r K_i$ be a CM algebra over $\Q$ and $E:=K^+$ the $\Q$-subalgebra fixed by the canonical involution $\iota$, and let $T=T^{K/E,\Q}$ be the associated CM torus over $\Q$.

\begin{thm}\label{I.3} 
For any integer $n$, there exists a CM torus $T$ over $\Q$ such that $\tau(T)=2^n$.
\end{thm}

Finally, we give a numerous of results of $\tau(T)$ for Galois CM fields. 

\begin{thm}\label{I.4} 
Suppose that $K$ is a Galois CM field with Galois group $G=\Gal(K/\Q)$. Let $g:=[K^+:\Q]$, $G^+=\Gal(K^+/\Q)$ and $T$ be the associated CM torus.
\begin{enumerate}
    \item[\rm (1)] If $K=\Q(\zeta_n)\neq \Q$ is the $n$th cyclotomic field with either $4|n$ or odd $n$, then  
\begin{enumerate} 
    \item[$(a)$] If $n$ is either a power of an odd prime $p$ or $n=4$, then $\tau(T) = 1$;
    \item[$(b)$] In other cases, we have $\tau(T) = 2$.
\end{enumerate}

\item[{\rm (2)}] If $G$ is abelian, then $\tau(T) \in \{1,2\}$.
Moreover, the following statements hold:
\begin{enumerate} 
    \item[$(a)$] If $g$ is odd, then $\tau(T) = 1$.
    \item[$(b)$] If $g$ is even and the exact sequence 
\[1 \rightarrow \langle\iota\rangle \rightarrow G \rightarrow G^{+} \rightarrow 1 \tag{$*$}\]
splits, then $\tau(T) = 2$.
\end{enumerate} 

\item[\rm (3)] Suppose that the short exact sequence $(*)$ splits. Then $\tau(T)\in \{1,2\}$ and the following statements hold:

\begin{enumerate}
    \item[{\rm (a)}] When $g$ is odd, $\tau(T) = 1$;
    \item[{\rm (b)}] Suppose $g$ is even and let $g^{ab}:=|G^{+{\rm ab}}|$, the cardinality of the abelianization of $G^+$.
\begin{enumerate} 
  \item[{\rm (i)}] If $g^{\rm ab}$ is even, then $\tau(T) = 2$;
  \item[{\rm (ii)}] If $g^{\rm ab}$ is odd, then there is a unique nonzero
          element $\xi$ in the $2$-torsion subgroup $H^{2}(\Lambda)[2]$ 
          of $H^{2}(\Lambda)$. Moreover, $\tau(T) = 1$ if and only if
          its restriction $r_D(\xi) = 0$ in $H^2(D,\Lambda)$ for all $D\in \scrD$. 
\end{enumerate}
\end{enumerate}

\item[{\rm (4)}] Let $P$ and $Q$ be two odd non-square positive integers such that $P-1=a^2$ and $Q-1=Pb^2$  
for some integers $a,b\in  \bbN$. Let $K:=\Q(\sqrt{\alpha})$ with $\alpha:=-(P+\sqrt{P})(Q+\sqrt{Q})$. Then  
$K$ is a  Galois CM field with Galois group $Q_8$ and 
\begin{equation}
    \label{eq:I.8}
    \tau(T)=\begin{cases}
    1/2, & \text{if $\left ( \frac{P}{q} \right) =1$ for all prime  $q|Q$;} \\
    2, & \text{otherwise}. 
    \end{cases}
\end{equation}   

\item[{\rm (5)}] Suppose the Galois extension $K/\Q$ has Galois group $D_n$ of order $2n$. Then $n$ is even and $\tau(T)=2$.
\end{enumerate}
\end{thm}

We explain the idea of the proof of Theorem~\ref{I.3}. First of all, it is rather difficult to construct a CM field such that the Tamagawa number $\tau(T)$ of the associated algebraic torus $T$ is small. Suppose that $K/\Q$ is Galois with Galois group $G$ and let $G_2$ be a Sylow $2$-group of $G$. T.~R\"ud implements a SAGE algorithm and has checked all $2$-groups of order $\le 128$ and the first 29631 groups of order 256 for $G_2$. Based on R\"ud's result there is at most one case such that $\tau(T)=1/4$; see \cite[Section 6.4]{rued:thesis}. 
To get around this, we construct an infinite family of "totally" linear disjoint $Q_8$-CM fields $\{K_i\}$ for $i\ge 1$ with $\tau(T_i)=1/2$, that is, the Galois group of the compositum of any finitely many of these $Q_8$-CM fields $K_i$ is the product of the Galois groups $\Gal(K_i/\Q)$.    
The CM algebra $K:=\prod_{i=1}^r K_i$ then satisfies the conditions in Theorem~\ref{I.2} and it follows that the CM torus $T$ associated to $K$ satisfies $\tau(T)=1/2^r$. 

We point out that the proof of Theorem~\ref{I.3} is actually quite
tricky. First of all, it follows from \cite{rued:thesis} or Theorem~\ref{I.4} that $Q_8$-CM fields are the simplest ones so that the associated CM tori $T$ can have $\tau(T)=1/2$. On the other hand, Theorem~\ref{I.2} requires a condition that every decomposition group of each CM field extension $K_i/\Q$ is cyclic. Fortunately, for any $Q_8$-CM field $K_i$ with CM torus $T_i$, one has $\tau(T_i)=1/2$ if and only if this condition for $K_i/\Q$ holds (see Proposition~\ref{prop:tauQ8}), so that we can apply Theorem~\ref{I.2} to the product of them. On other other hand, one may be wondering whether this condition is superfluous. For this question, we construct linearly disjoint Galois CM fields $K_1$ and $K_2$ such that
\begin{equation}\label{eq:count_eg}
    \tau(T_1)=2, \quad \tau(T_2)=1, \quad \tau(T)=1,
\end{equation}
where $T_1$, $T_2$ and $T$ are CM tori associated to $K_1$, $K_2$ and
$K_1\times K_2$, respectively. 
This example shows that the cyclicity of decomposition groups of every $K_i/\Q$ is not superfluous. 


Theorem~\ref{I.3} proves an analogous but more involved conjecture of Ono for CM tori.  
In the appendix Jianing Li and the fourth named author show that for any global field $k$ and any positive rational number $\alpha$, there exists an algebraic torus $T$ over $k$ with $\tau(T)=\alpha$; see Theorem~\ref{thm:ono}. This extends the result of Katayama ~\cite{katayama:1985} and proves Ono's conjecture for global fields.

Though our original motivation of investigating Tamagawa numbers of CM tori comes from counting certain abelian varieties and exploring the structure of CM fields, the main part of the problem itself was to compute the Tate-Shafarevich group. 
We explain in Remark~\ref{rem:ker1} how the Tate-Shafarevich group of a CM torus also comes into play in the theory of Shimura varieties of PEL-type. 
Tate-Shafarevich groups measure the failure of the local-global principle that is one of main interests in number theory and has been actively studied. H\"urlimann~\cite{hurlimann} proved that the multiple norm principle holds for any etale $k$ algebra $K_1\times K_2$ in which one component is cyclic and the other one is Galois.
Prasad and Rapinchuk \cite{prasad-rapinchuk} settled the problem of the local-global principle for embeddings of fields with involution into simple algebras with involution, where they also investigated the multiple norm principle. The multiple norm principle has been investigated further by Pollio and Rapinchuk~\cite{pollio-rapinchuk, pollio},  Demarche and D.~Wei~\cite{demarche-wei} and D.~Wei and F.~Xu~\cite{wei-xu:int_pt}. Bayer-Fluckiger, T.-Y.~Lee and Parimala \cite{BLP19} studied the Tate-Shafarevich group of general multinorm one tori in which one of factors is a cyclic extension. They give a simple rule for determining the $\Sha$-group in the case of products of extensions of prime degree $p$. T.-Y. Lee~\cite{lee:jpaa} computes explicitly the $\Sha$-group for the cases where every factor is cyclic of degree $p$-power.      

This article is organized as follows. Section~\ref{sec:P} includes preliminaries and background on Tamagawa numbers of algebraic tori, and some known results of those of CM tori due to Li--R\"ud and Guo--Sheu--Yu. Section~\ref{sec:T} discusses transfer maps, their extensions and connection with class field theory. In Sections~\ref{sec:C} and~\ref{sec:MC}
we compute the Galois cohomology groups of character groups of a class of algebraic tori $T^{K/E,k}$ and $T^{K/E,1}$. Section~\ref{sec:SG} treats Galois CM tori and in Sections~\ref{sec:NS} ans \ref{sec:Q8} we determine the precise ranges of Tamagawa numbers of CM tori. In Section~\ref{sec:P2} we show there are infinitely many pairs of linearly disjoint Galois CM fields $K_1$ and $K_2$ satisfying \eqref{eq:count_eg}. 
In the appendix Jianing Li and the fourth author prove Ono's conjecture for global fields.   

\section{Preliminaries, background and some known results}
\label{sec:P}

\subsection{Class numbers and Tamagawa numbers of algebraic tori}
\label{sec:P.1}
The cardinality of a set $S$ will be denoted by $\abs{S}$ or $\# S$. 
For any field $k$, let $\bar k$ be a fixed algebraic closure of $k$,
let $k^{\sep}$ be the separable closure of $k$ in $\bar k$ and denote by $\Gamma_k:=\Gal(k^{\sep}/k)$ the Galois group of $k$.
Let $\G_{{\rm m},k}:=\Spec k[X,X^{-1}]$ denote the
multiplicative group associated to $k$ with the usual multiplicative group law.

\begin{defn}
  (1) A connected linear algebraic group $T$ over a field $k$ is said to be an \emph{algebraic torus} over  $k$ if there exists a finite field extension
  $K/k$ such that there exists an isomorphism 
  $ T \otimes_k K \simeq \GmK^d$ of algebraic groups over $K$ for some positive integer $d$. Then $d$ is equal to 
the dimension of $T$. If $T$ is an algebraic $k$-torus and $K$ is a field extension of $k$ that satisfies the  above property,    
then $K$ is called a \emph{splitting field}
of $T$. The smallest splitting field (which is unique and Galois, see below) is
called \emph{the minimal splitting field of $T$}.

  (2) For any algebraic torus $T$ over $k$, denote by $X(T):=\Hom_{ k^{\rm sep}}(T\otimes_{k}  k^{\rm sep}, \G_{{\rm m}, k^{\rm sep}})$ the group of characters of $T$.
It is a finite free $\Z$-module of rank $d$ together with a continuous action of the Galois group $\Gamma_k$ of $k$.  
\end{defn}

It is shown in \cite[Proposition 1.2.1]{ono:arithmetic}
(also see \cite{yu:chow} for other proofs) that every algebraic $k$-torus splits over a finite \emph{separable} field extension $K/k$. The action of $\Gamma_k$ on $X(T)$ gives a continuous representation
\begin{equation}
  \label{eq:rT}
  r_T: \Gamma_k \to \Aut(X(T))
\end{equation}
which factors through a faithful representation of a finite quotient $\Gal(L/k)$ of $\Gamma_k$.
Here $L$ is the fixed field of the kernel of $r_T$ and is the smallest splitting field of $T$. In particular, $L$ is  a finite Galois extension of $k$ and $X(T)$ can be also regarded as a $\Gal(L/k)$-module.

From the remainder of this article,
$k$ denotes a global field and a finite extension $K/k$ will be separable, unless otherwise stated.
For each place $v$ of $k$, denote by
$k_v$ the completion of $k$ at $v$, and $O_{v}$ the ring of integers of $k_v$ if $v$ is finite. For each finite place $v$, the group $T(k_v)$ of $k_v$-valued points of $T$ contains a unique maximal open compact subgroup, which is denoted by $T(O_v)$. For a non-empty finite set $S$ of places of $k$ containing all non-Archimedean places if $k$ is a number field. 
Let $\A_k$ be the adele ring of of $k$.
Denote by $U_{T,S}=\prod_{v\in S} T(k_v) \times \prod_{v \notin S} T(O_v)$ the unit group with respect to $S$
and let $\Cl_S(T):=T(\A_k)/T(k) U_{T,S}$ be the \emph{$S$-class group} of $T$. By a finiteness theorem of Borel~\cite{Borel-finiteness}, $\Cl_S(T)$ is a finite group and its cardinality is denoted by $h_S(T)$, called the \emph{$S$-class number} of $T$.
If $k$ is a number field and $S=\infty$ consists of all non-Archimedean places, we write $U_T:=U_{T,\infty}$, $\Cl(T):=T(\A_k)/T(k)U_T$ the \emph{class group} of $T$ and call $h(T):=|\Cl(T)|$ the \emph{class number} of $T$.

It follows immediately from the definition that if $K/k$ is a finite extension and $T_K$ is an algebraic $K$-torus, then $h(R_{K/k}T_K)=h(T_K)$, where $R_{K/k}$ denotes the Weil restriction of scalars from $K$ to $k$. Note that
$\dim R_{K/k} T_K = \dim T_K \cdot [K:k]$. It is well known (cf. \cite{ono:arithmetic,ono:tamagawa}) that
\begin{equation}
  \label{eq:ind_char}
  X(R_{K/k}T_K)\simeq \Ind_{\Gamma_K}^{\Gamma_k} X(T_K)=\Ind_{\Gal(L/K)}^{\Gal(L/k)} X(T_K),
\end{equation}
as $\Gal(L/k)$-modules,  where $L$ is any finite Galois extension of $k$ over which the algebraic $k$-torus $R_{K/k}T_K$ splits.
 
We recall the definition of the Tamagawa number of an algebraic $k$-torus $T$.
Fix a finite Galois splitting field extension $K/k$ of $T$ with Galois group $G$.
Let $ \chi _T $ be the character of representation $ \big( X(T)\otimes {\C} , r_T \big) $ of $G$ over $\mathbb{C}$, that is,  $ \chi _T : G \rightarrow \mathbb{C} $, $ \chi _T(g) := \tr ( r_T ( g ) ) $ for all $g\in G$.
Let
\begin{equation}
  \label{eq:artinL}
L( s , K/k , \chi _T ) := \prod_{ v \nmid \infty } L_v( s , K/k , \chi _T )  
\end{equation}
be the Artin $L$-function of the character $\chi_T$, where
$L_v(s,K/k,\chi_T)$ is the local Artin $L$-factor at $v$.
It is well known (cf.~\cite{Serre-linear})  
that
$L( s , K/k , \chi _T )$ has a pole of order $a$ at $s=1$, where $a$ is the rank of the $G$-invariant sublattice $X(T)^G$.

Let $\omega$ be a nonzero invariant differential form on $T$ of highest degree defined over $k$. 
To each place $v$, one associates a Haar measure $\omega_v$ on
$T(k_v)$. Then the product of the Haar measures
\[  \prod_{ v \mid \infty } \omega_v \cdot \prod_{ v \nmid \infty }
  \left ( L_v( 1 , K/k , \chi _T ) \cdot \omega_v \right ) \]
converges absolutely and defines a Haar measure on $T(\A_k)$.

Write (N) for the number field case and (F) for the function field case.
For (N), let $d_k$ be the discriminant of $k$. For (F), $k=\Fq(C)$ is the function field of a smooth projective geometrically connected curve over $\Fq$ and let $g(C)$ be the genus of $C$. 

\begin{defn} Let $T$ be an algebraic torus over a global field $k$ and $\omega$ be a nonzero invariant differential form on $T$ defined over $k$ of highest degree. Then
\[
  \omega_{\mathbb{A} , can } := \frac{ \prod_{ v \mid \infty } \omega_v \cdot
  \prod_{ v \nmid \infty } \big( L_v( 1 , K/k , \chi _T )\cdot \omega_v \big) }{ \mu_k^{ (\dim{T}) } \cdot \rho_T },
\]
defines a Haar measure on $T(\A_k)$, which is called
the \emph{Tamagawa measure on $T(\mathbb{A}_k)$},  
where 
\[ \mu_k:=\begin{cases}
|d_k|^{1/2},&  \text{for (N);} \\
q^{g(C)}, & \text{for (F),} 
\end{cases} \]
\[ \rho_T := \lim_{ s \to 1 } ( s - 1 )^a L( s , K/k , \chi _T ), \]
and $a$ is the order of the pole of the Artin $L$-function
$L( s , K/k , \chi _T )$ at $ s = 1 $.  
\end{defn} 

Let $\xi_1, \dots ,\xi_a$ be a basis of $X(T)^G$. Define
\[ \xi: T(\A_k) \to \R_{+}^{a}, \quad x\mapsto 
(\vert\vert \xi_1(x)||,\dots, \vert \vert\xi_{a}(x)\vert\vert), \]
where $\R_+:=\{x\in \R^\times: x>0 \}\subset \R^\times$ is the connected open subgroup. 
Let $T(\A_k)^1$ denote the kernel of $\xi$; one has an isomorphism 
\[ T(\A_k)/T(\A_k)^1 \simeq  
\begin{cases} 
\R_{+}^a, & \text{for (N);} \\
(q^\Z)^a, & \text{for (F)}.
\end{cases}
\]
Let $d^\times t:=\prod_{i=1}^{a} dt_i/t_i$ be the canonical measure on $\R_{+}^a$ for (N) and $d^\times t$ be the counting measure on $(q^\Z)^a$ with measure $(\log q)^a$ on each point for (F). 
Let $\omega^1_{\A,can}$ be the unique Haar measure on $T(\A_k)^1$ such that 
\begin{equation}
  \label{eq:omega1}
  \omega_{\A,can}=\omega^1_{\A,can}\cdot d^\times t,
\end{equation}
that is, for any measurable function $F$
on $T(\A_k)$ one has
\[ \int_{T(\A_k)/T(\A_k)^1} \int_{T(\A_k)^1} F(xt) \,
\omega^1_{\A,can}\cdot d^\times t=\int_{T(\A_k)} F(x)\,
\omega_{\A,can}. \]
By a well-known theorem of Borel and Harish-Chandra \cite[Theorem
  5.6]{platonov-rapinchuk}, the quotient
  space $T(\A_k)^1/T(k)$ has finite volume with respect to every Haar
  measure. In fact $T(\A_k)^1/T(k)$ is the unique maximal
  compact subgroup of $T(\A_k)/T(k)$, because $\R_{+}^a$
  has no nontrivial compact subgroup.

\begin{defn} Let $T$ be an algebraic torus over a global field $k$.
The \emph{ Tamagawa number $\tau_k(T)$ of $T$} 
is defined by
\begin{equation}\label{eq:tau}
\tau_k(T):=\int_{T(\A_k)^1/T(k)}\ \omega^1_{\A,can},
\end{equation}
the volume of $T(\A_k)^1/T(k)$ with respect to $\omega^1_{\A,can}$, where
$\omega^1_{\A,can}$ is the Haar measure on $T(\A_k)^1$ defined in \eqref{eq:omega1}. 
\end{defn} 

One has the following properties (\cite[Theorem 3.5.1]{ono:arithmetic} cf.~\cite[Section 3.2, p.~57]{ono:tamagawa}):
\begin{enumerate} 
\item[{\rm (i)}] For any two algebraic $k$-tori $T$ and $T'$, one has $ \tau_k(T \times T') = \tau_k(T) \cdot \tau_k(T') $.
\item[{\rm (ii)}] For any finite extension $K/k$ and any algebraic $K$-torus $T_K$, one has  $\tau_k(R_{K/k}T_K) = \tau_K(T_K) $.
\item[{\rm (iii)}] One has $\tau_k({\Gm}_{,k})=1 $.
\end{enumerate}

Note that the last statement (iii) is equivalent to the analytic class number formula \cite[VIII \S 2 Theorem 5, p.~161]{Lang-ANT}.

\subsection{Values of Tamagawa numbers}\label{sec:P.2}

\begin{thm}[Ono's formula]
Let $K/k $ be a finite Galois extension with Galois group $\Gamma$, and $T$ be an algebraic torus over $k$ with splitting field $K$. Then
\begin{equation}
    \label{eq:ono-1}
   \tau_k (T) = \frac{| H^1 \big( \Gamma , X(T) \big) |}{| \Sha^1(\Gamma, T)|},   
\end{equation}
where
\begin{equation}
  \label{eq:ShaT}
\Sha^i(\Gamma, T):=\Ker \big( H^i (K/k , T) \to \prod_v H^i (K_w/k_v , T) \big)  
\end{equation}
is the Tate-Shafarevich group associated to $H^i (\Gamma , T)$ for $i\ge 0$ and $w$ is a place of $K$ over $v$.
\end{thm}
\begin{proof}
  See \cite[Section 5]{ono:tamagawa}. \qed 
\end{proof}

\begin{rem}\label{rem:indepXT}
  The cohomology groups $H^1 \big( \Gamma , X(T) \big)$ and
  $\Sha^1(\Gamma, T)$ are independent of the choice of the splitting
  field $K$;  see \cite[Sections 3.3 and 3.4]{ono:tamagawa}.
\end{rem}

According to Ono's formula, the Tamagawa number of any algebraic $k$-torus  is a positive rational number. Ono constructed an infinite family of algebraic $\Q$-tori $T$ with $\tau(T)=\tau_{\Q}(T)=1/4$,  which particularly shows that $\tau(T)$ needs not to be an integer. Ono conjectured \cite{ono:sugaku}
that every positive rational number can be realized as $\tau_k(T)$ for some algebraic $k$-torus $T$. 
Ono's conjecture was proved by S. Katayama~\cite{katayama:1985} for the number field case.

For any abelian group $G$, the Pontryagin dual of $G$ is defined to be
\begin{equation}
  \label{eq:Pdual}
   G^\vee:=\Hom(G,\Q/\Z).
\end{equation}
The Poitou-Tate duality (\cite[Theorem 6.10]{platonov-rapinchuk} and \cite[Theorem 8.6.8]{NSW})
says that there is a natural isomorphism $\Sha^1(\Gamma, T)^\vee\simeq  \Sha^2(\Gamma, X(T))$. Thus, 
\begin{equation}
    \label{eq:ono-f-2}
       \tau_k (T) = \frac{| H^1 \big( \Gamma , X(T) \big) |}{| \Sha^2(\Gamma, X(T))|}.   
\end{equation}

To simply the notation we shall often 
suppress the Galois group from Galois
cohomology groups and write $H^1(X(T))$ and $\Sha^2(X(T))$ for
$H^1(\Gamma, X(T))$ and $\Sha^2(\Gamma, X(T))$ etc.,
if there is no risk of confusion. 

\subsection{Some known results for CM tori}
\label{sec:P.3}
For the convenience of later generalization and investigation, we define here a more general class of $k$-tori as mentioned in Introduction. 

For every commutative 
etale $k$-algebra $K$, 
denote by $T^K$ the algebraic $k$-torus whose group of $R$-valued points of $T^K$
for any commutative $k$-algebra $R$, is 
\[ T^K(R)=(K \otimes_k R)^\times. \]
Explicitly, if $K=\prod_{i=1}^r K_i$ is a product of finite separable field
extensions $K_i$ of $k$,
then 
\[ T^K=\prod_{i=1}^r T^{K_i}=\prod_{i=1}^r R_{K_i/k} (\G_{m,K_i}), \] 
where $R_{K_i/k}$
is the Weil restriction of scalars from $K_i$ to $k$. 
Let $E=\prod_{i=1}^r E_i$ be a product of finite subfield extensions
$E_i\subset K_i$ over $k$. Let $N_{K_i/E_i}: T^{K_i} \to T^{E_i}$ 
be the norm map and put $N=\prod_{i=1}^r N_{K_i/E_i}: T^K \to T^E$. Define 
$T^{K/E,1}:=\Ker N$, the kernel of the norm map $N$, and 
\[ T^{K/E,k}:=\{x\in T^K : N(x)\in \G_{m,k} \}, \]
the preimage of $\G_{m,k}$ in $T^K$ under $N$, 
where $\G_{m,k}\embed T^E$ is viewed as a subtorus of $T^E$ via the diagonal
embedding. Then we have the following commutative diagram of algebraic $k$-tori in which each row is an exact sequence:

\begin{equation}\label{eq:diagram_tori}
\begin{tikzcd}
\qquad 1 \arrow{r} & T^{K/E,1}\arrow{r}{j}\arrow[equal]{d} 
& T^K\arrow{r}{N}& T^{E}\arrow{r} & 1 \\
\qquad 1 \arrow{r} & T^{K/E,1} \arrow{r}{j} & T^{K/E,k}\arrow{r}{N}
\arrow[hook]{u} & \G_{{\rm m},k}\arrow[hook]{u}{}\arrow{r} & 1.
\end{tikzcd}
\end{equation}

For the rest of this subsection we let $k=\Q$ and $K=\prod_{i=1}^r K_i$ be a CM algebra, where each $K_i$ is a CM field, with the canonical involution $\iota$. 
The subalgebra $K^+\subset K$ fixed by $\iota$ is the product $K^+=\prod_{i=1}^rK_i^+$ of the maximal totally real subfields $K_i^+$ of $K_i$.  
Let $T^{K,1}:=T^{K/K^+,1}=\Ker N_{K/K^+}$ be the associated norm one CM torus and $T^{K,\Q}:=T^{K/K^+,\Q}$ the associated CM torus, respectively. As before, we have the following exact sequence of algebraic tori over $\Q$
\begin{equation}\label{exact sequence of T^K}
\begin{tikzcd}
    1 \arrow{r} & T^{K,1}\arrow{r} & T^K\arrow{r}{N_{K\slash K^+}} & T^{K^+}\arrow{r} & 1,
\end{tikzcd}
\end{equation}
and a commutative diagram similar to \eqref{eq:diagram_tori}.

\begin{prop}\label{formula}
Let $K$ be a CM algebra and $T=T^{K,\Q}$ the associated CM torus over $\Q$. 

{\rm (1)} We have
\begin{equation}
  \label{eq:CMformula}
  \tau(T)=\frac{2^r}{n_K},
\end{equation}
where $r$ is the number of components of $K$ and
\begin{equation}
  \label{eq:nK}
  n_K:=[\A^\times: N(T(\A))\cdot \Q^\times]
\end{equation}
is the global norm index associated to $T$. 

{\rm (2)} We have
\begin{equation}
    \label{eq:bound_nK}
    n_K\, \Big{|}\!  \prod_{p\in S_{K/K^+}}\!\!\!\! e_{T,p},
\end{equation}  
where $S_{K/K^+}$ is the finite set of rational primes $p$
with some place $v|p$ of $K^+$ ramified in $K$, and 
$e_{T, p}:=[\Z^\times_p: N(\,T(\Z_p)\,)]$. Here $T(\Zp)$ denotes the unique maximal open compact subgroup of $T(\Qp)$.
\end{prop}
\begin{proof}
  (1) This is \cite[Theorem 1.1(1)]{guo-sheu-yu:CM}. (2) This is \cite[Lemma 4.6(2)]{guo-sheu-yu:CM}. \qed
\end{proof}

\begin{lemma}\label{lm:E}
  Let $K$ and $T$ be as in Proposition~\ref{formula}.

 {\rm (1)} If $K$ contains an imaginary quadratic field, then $n_K\in \{1,2\}$ and $\tau(T)\in \{ 2^{r-1}, 2^r\}$. 

 {\rm (2)} If $K$ contains two distinct imaginary quadratic fields, then $n_K=1$ and $\tau(T)=2^r$.    
\end{lemma}
\begin{proof}
  This is proved in \cite[Lemma 4.7]{guo-sheu-yu:CM} (cf.~\cite[Section 5.1]{guo-sheu-yu:CM}) in the case where $K$ is a CM field.
  The same proof using class field theory also proves the CM algebra case. \qed 
\end{proof}

\begin{prop}\label{LiRu}
  Let $K$ be a CM field and $T$ the associated CM torus over $\Q$. Put $g=[K^+:\Q]$ and let $K^{\Gal}$ be the Galois closure of $K$ over $\Q$ with Galois group $\Gal(K^{\Gal}/\Q)=G$. 

{\rm (1)} If $g$ is odd, then $H^1(X(T))=0$.

{\rm (2)} If $K/\Q$ is Galois and $g$ is odd, then $\tau(T)=1$.

{\rm (3)} If $K/\Q$ is cyclic, then $\tau(T)=1$.

{\rm (4)} If $K/\Q$ is Galois of degree $4$, then $\tau(T)\in\{1,2\}$. Moreover, $\tau(T)=2$ if and only if $G \simeq (\Z/2\Z)^2$.  
\end{prop}
\begin{proof}
  See Propositions A.2 and A.12 of \cite{achter-altug-gordon}. \qed 
\end{proof}

Note that Proposition~\ref{LiRu} (4) also follows from Proposition~\ref{LiRu}(3) and Lemma~\ref{lm:E} (2). 

\section{Transfer maps, corestriction maps and extensions}
\label{sec:T}

\subsection{Transfer maps}
\label{sec:T.1}

Let $G$ be a group and let $H$ be a subgroup of $G$ of finite index. Let $X:=G/H$, and let $\varphi: X\to G$ be a section. If $g\in G$ and $x\in X$, the elements $\varphi(gx)$ and $g\varphi(x)$ belong to the same class of mod $H$; hence there exists a unique element $h^\varphi_{g,x}\in H$ such that  $g\varphi(x)=\varphi(gx)h^\varphi_{g,x}$. Let $\Ver_{G,H}(g)\in H^{\rm ab}$ be defined by:
\[ \Ver_{G,H}(g):=\prod_{x\in X} h^\varphi_{g,x} \mod D(H), \]
where $D(H)=[H,H]$ is the commutator group of $H$ and the product is computed in $H^{\rm ab}=H/D(H)$.

By \cite[Theorem 7.1]{Serre-finite-grp}, the map $\Ver_{G,H}: G\to H^{\rm ab}$ is a group homomorphism and it does not depend on the choice of the section $\varphi$. This homomorphism is called the \emph{transfer} of $G$ into $H^{\rm ab}$ (originally from the term  ``Verlagerung'' in German). One may also view it as a homomorphism $\Ver_{G,H}:G^{\rm ab}\to H^{\rm ab}$.

In the literature one also uses the right coset space $X':=H\backslash G$ but this does not effect the result. One can easily show that if $\Ver_{G,H}'$ is the transfer map defined using $X'$, then $\Ver_{G,H}'=\Ver_{G,H}$.
One has the following functorial property (see~\cite[p.~89]{Serre-finite-grp}).

\begin{lemma}\label{lm:func}
  Let $H\subset G$ and $H'\subset G'$ be subgroups of finite index.  
  If $\sigma$ be a group homomorphism from the pair $(G,H)$ to $(G',H')$ which induces a bijection $G/H\isoto G'/H'$, then the following diagram  
\begin{equation} \label{eq:func}
\begin{tikzcd}
  \qquad  & G^{\rm ab} \arrow{r}{\sigma} \arrow{d}{\Ver_{G,H}}
  & G^{' \rm ab} \arrow{d}{\Ver_{G'.H'}} \\  
\qquad & H^{\rm ab} \arrow{r}{\sigma} \arrow{r} & H^{' \rm ab}
\end{tikzcd}  
\end{equation}
commutes.
\end{lemma}

\begin{lemma}\label{lm:product}
  Let $G_1$ and $G_2$ be groups, and let $H_i\subset G_i$ be a subgroup of finite index for $i=1,2$. Then
\[  \Ver_{G_1\times G_2, H_1\times H_1}(g_1,g_2)=\Ver_{G_1,H_1}(g_1)^{[G_2:H_2]} \cdot \Ver_{G_2,H_2}(g_2)^{[G_1:H_1]}. \]
\end{lemma}

\begin{proof}
  Put $G=G_1\times G_2$, $H=H_1\times H_2$, $X_i:=G_i/H_i$ and $X=G/H=X_1\times X_2$. Fix a section $\varphi_i:X_i \to G_i$ for each $i$ and let $\varphi=(\varphi_1,\varphi_2):X\to G$. For $g=(g_1,g_2)$ and $x=(x_1,x_2)$, we have $h^\varphi_{g,x}=(h^{\varphi_1}_{g_1,x_1},h^{\varphi_2}_{g_2,x_2})$. Then in $H^{\rm ab}=H_1^{\rm ab} \times H_2^{\rm ab}$ we have 
  \[
    \begin{split}
      \Ver_{G,H}(g)&=\prod_{x\in X} h^\varphi_{g,x}=\prod_{x_1\in X_1} \prod_{x_2\in X_2} (h^{\varphi_1}_{g_1,x_1}, h^{\varphi_2}_{g_2,x_2})\\
      &=\prod_{x_1\in X_1} ((h^{\varphi_1}_{g_1,x_1})^{|X_2|}, \Ver_{G_2,H_2}(g_2))= (\Ver_{G_1,H_1}(g_1)^{|X_2|}, \Ver_{G_2,H_2}(g_2)^{|X_1|}). \quad \text{\qed}
    \end{split}  \]    
\end{proof}

If $G$ is a finite group and $p$ is a prime, $G_p$ denotes a $p$-Sylow subgroup of $G$.

\begin{prop}\label{prop:rued}
  Let $G$ be a finite group and $N\triangleleft G$ be a cyclic normal subgroup of prime order $p$. Then the transfer $\Ver_{G,N}:G^{\rm ab}\to N^{\rm ab}$ is surjective if and only if $G_p$ is cyclic.  
\end{prop}
\begin{proof}
 See Proposition 3.8 of \cite{rued:thesis}.
 Note that the author assumes that the normal subgroup $N$ is central in $G$ throughout the whole
 paper; however, this assumption is not needed for
 this proposition in his proof.
 \qed 
\end{proof}

\subsection{Connection with corestriction maps}
\label{sec:T.2}

Let $H \subset G$ be a subgroup of finite index, and let $A$ be a $G$-module.
Let
\[ f: \Ind_H^G A=\Z[G]\otimes_{\Z[H]} A \to A, \quad g\otimes a \mapsto ga \]
be the natural map of $G$-modules. Applying Galois cohomology $H^i(G, -)$ to the map $f$ and by Shapiro's lemma, we obtain  for each $i\ge 0$ a morphism
\[ \Cor: H^i(H, A) \to H^i(G,A), \]
called the \emph{corestriction} form $H$ to $G$.

Applying $H^i(G,-)$ to the short exact sequence
$0 \longrightarrow \Z  \longrightarrow \Q  \longrightarrow \Q/\Z  \longrightarrow 0$,
one obtains an isomorphism
$ H^i(G,\Z) \simeq H^{i-1}(G, \Q/\Z)$ 
for all $i\ge 1$. When $i=2$, this gives the isomorphism
\begin{equation}
  \label{eq:H2Z}
  H^2(G,\Z)\simeq \Hom(G^{\rm ab},\Q/\Z).
\end{equation}
\begin{prop} \label{prop:cor}
  Let $H\subset G$ be a subgroup of finite index. Through the isomorphism \eqref{eq:H2Z} we have the following commutative diagram
\begin{equation} \label{eq:cor-ver}
\begin{tikzcd}
  \qquad  & H^2(H,\Z) \arrow{r}{\sim} \arrow{d}{\Cor}
  & \Hom(H^{\rm ab},\Q/\Z) \arrow{d}{\Ver_{G,H}^\vee} \\  
\qquad & H^2(G,\Z) \arrow{r}{\sim} \arrow{r} & \Hom(G^{\rm ab},\Q/\Z), 
\end{tikzcd}
\end{equation}
where $\Ver_{G,H}^{\vee}$ is dual of the transfer $\Ver_{G,H}: G^{\rm ab} \to H^{\rm ab}$.
Moreover, if $H$ is normal in $G$, the composition $H^{\rm ab}\to G^{\rm ab}\to H^{\rm ab}$ is the norm $N_{G/H}$. Here $H^{\rm ab}$ is viewed as a $G$-module by conjugation and also as a $G/H$-module, since $H$ acts trivially on
$H^{\rm ab}$.    
\end{prop}
\begin{proof}
  See \cite[Proposition 1.5.9]{NSW}.
\end{proof} \

\subsection{Connection with class field theory}
\label{sec:T.3}

Let $k\subset K \subset L$ be three global fields such that the extension $L/k$ is Galois. 
Put $G=\Gal(L/k)$ and $H=\Gal(L/K)\subset G$.
Denote by $C_k$ and $C_K$ the idele class groups of $k$ and $K$, respectively.
The Artin map is a surjective homomorphism $\Art_{L/k}: C_k \to G^{\rm ab}$; similarly we have $\Art_{L/K}: C_K \to H^{\rm ab}$.
By class field theory we have the following commutative diagrams
\begin{equation}
  \label{eq:cft}
  \begin{CD}
    C_k @>\Art_{L/k}>> G^{\rm ab} \\
    @V{\rm func}VV @VV{\Ver_{G,H}}V \\
    C_K @>\Art_{L/K}>> H^{\rm ab} \\ 
  \end{CD} \qquad \qquad
   \begin{CD}
    C_k @>\Art_{L/k}>> G^{\rm ab} \\
    @A{N_{K/k}}AA @AA{\rm func}A \\
    C_K @>\Art_{L/K}>> H^{\rm ab}, \\ 
  \end{CD}
\end{equation}
where ${\rm func}$ denotes the natural map induced from the inclusion $\A_k^\times \embed \A_K^\times$ or $H\embed G$. When $K/k$ is Galois and let 
$L=K$,
the second diagram of \eqref{eq:cft} induces a homomorphism
\[ \Art_{K/k}: C_k/N_{K/k}(C_K) \to \Gal(K/k)^{\rm ab}, \] which is an isomorphism by class field theory. 

\subsection{Relative transfer maps}
\label{sec:T.4}

Let $H\subset N \subset G$ be two subgroups of $G$ of finite index.
Let $\wt X:= G/H$ and $X:=G/N$ with natural $G$-equivariant projections
$\wt c: G\to X$ and $c: \wt X\to X$. Let $\wt \varphi: X\to G$ be a section, which induces a section $\varphi: X\to \wt X$. For each $g\in G$ and $x\in X$, let $n^{\wt \varphi}_{g,x}$ be the unique element in $N$ such that $g \wt \varphi (x)=\wt \varphi (gx) n^{\wt \varphi}_{g,x}$.
Let $(N/H)^{\rm ab}=N/D(N)H$ denote the maximal abelian quotient of $N$ which maps $H$ to zero.  
Let $\Ver_{G,N/H}(g) \in (N/H)^{\rm ab}$ be the element defined by
\begin{equation}
  \label{eq:rt}
  \Ver_{G,N/H}(g):=\prod_{x\in X} n^{\wt \varphi}_{g,x} \mod D(N)H.
\end{equation}

\begin{prop}\label{prop:rt}
  {\rm (1)} The map $\Ver_{G,N/H}\colon G\to (N/H)^{\rm ab}$ does not depend on the choice of the section $\wt \varphi$ and it is a group homomorphism.

  {\rm (2)} One has $\Ver_{G,N/H}=\pi_H \circ \Ver$, where $\pi_H\colon N^{\rm ab} \to (N/H)^{\rm ab}$ is the morphism mod $H$. 
\end{prop}
\begin{proof}  
  Clearly, the statement (1) follows from (2), because $\Ver$ does not depend on the choice of $\wt \varphi$ and is a group homomorphism. 
  (2) By definition
  $\Ver(g)=\prod_{x\in X} n^{\wt \varphi}_{g,x} \mod D(N)$, thus $\Ver_{G,N/H}=\pi_H \circ
  \Ver$. \qed
\end{proof}

\begin{defn}\label{def:rt}
  Let $H\subset N$ be two subgroups of $G$ of finite index. The group homomorphism $\Ver_{G,N/H}: G\to (N/H)^{\rm ab}$ defined in \eqref{eq:rt} is called the \emph{transfer} of $G$ into $(N/H)^{\rm ab}$ relative to $H$. By abuse of notation, we denote the induced map by $\Ver_{G,N/H}: G^{\rm ab} \to (N/H)^{\rm ab}$.
\end{defn}

One can check directly that the map $\Ver_{G,N/H}: G^{\rm ab} \to (N/H)^{\rm ab}$ factors through $\pi_H: G^{\rm ab} \to (G/H)^{\rm ab}$, the map modulo $H$.
We denote the induced map by 
\begin{equation}\label{eq:rV} 
  \Ver_{G/H,N/H}: (G/H)^{\rm ab} \to (N/H)^{\rm ab}.  
\end{equation}

\begin{rem}\label{rm:Hnormal}
  If $H\triangleleft G$ is a normal subgroup, then the induced map $\Ver_{G/H,N/H}$ is the transfer map from $G/H$ to $N/H$ associated to the subgroup $N/H\subset G/H$ of finite index.  
\end{rem}


\begin{lemma}\label{lm:ver-dual}
  Let $H\triangleleft \wt  N$ be two subgroups of finite index in $G$ with $H$ normal in $\wt N$ and cyclic quotient $N=\wt N/H=\< \sigma \>$ of order $n$. Fix a lift $\wt \sigma\in \wt N$ of $\sigma$. 
  
  {\rm (1)} Let $g\in G$ and $\{x_1,\dots x_r\}$ be double coset representatives of $\<g\>\backslash G/ \wt N$. Then 
\[ \Ver_{G,N}(g)=\sigma^{\sum_{i=1}^r m(g,x_i)}, \] 
where $f_i:=|\<g\> x_i H/H|$ and $0\le m(g,x_i)\le n-1$ is the unique integer such that $g^{f_i} x_iH=x_i \wt \sigma^{m(g,x_i)} H$.

  {\rm (2)} Let $\sigma^*\in N^{\vee}:=\Hom(N,\Q/\Z)$ be the element defined by $\sigma^*(\sigma)=1/n \mod \Z$,  $\Ver_{G,N}^\vee: N^\vee\to \Hom(G,\Q/\Z)$ the induced map, and $f:=\Ver_{G,N}^\vee(\sigma^*)$. Then $f(g)=\left (\sum_{i=1}^r m(g,x_i) \right )/n \mod \Z$.
\end{lemma}
\begin{proof}
(1) We choose the set of  representatives $S=\{g^j x_i: i=1,\dots, r,\  j=0,\dots, f_i-1\}$ of $X=G/\wt N$, which defines a section $\wt \varphi$ of the natural projection $G\to G/\wt N$. Then the image of the element $n^{\wt \varphi}_{g, x}$ in $H$ is given by
\[ n^{\wt \varphi}_{g, x} \mod H=\begin{cases} \sigma^{m(g,x_i)} & \text{if $x=g^{f_i-1} x_i H$ for some $1\le i\le r$;} \\
1 & \text{otherwise.}  \
\end{cases} \]

(2) By definition, $f(g)=\sigma^*(\Ver_{G,N}(g))=\sigma^*(\sigma^{\sum_i m(g,x_i)})=\left (\sum_i m(g,x_i)\right )/n \mod \Z$. \qed
\end{proof}

One can show that the integer $m(g,x_i)$ is independent of the choice of double coset representative in $\<g\> x_i H$.

\begin{lemma}\label{lm:prod_ver}
  Let $G=\prod_{i=1}^r G_i$ be a product of groups $G_i$ and let $N_i\subset G_i$, for each $1\le i \le r$, be a subgroup of finite index. Put 
  \[ H_i:=G_1\times \cdots \times G_{i-1}\times \{1\} \times G_{i+1}\times \cdots \times G_r  , \quad \wt N_i:=G_1\times \cdots  \times G_{i-1}\times N_i \times G_{i+1}\times \cdots \times G_r. \] 
  Then the map 
  \begin{equation}\label{eq:prod_ver-1}
     \prod_{i=1}^r \Ver_{G, \wt N_i/H_i}: G \longrightarrow (\wt N_1/H_1)^{\rm ab}\times \cdots \times (\wt N_r/H_r)^{\rm ab}=N_1^{\rm ab} \times \cdots \times N_r^{\rm ab}
  \end{equation}
  is given by the product of the maps
  \begin{equation}\label{eq:prod_ver-2}
      \prod_{i=1}^r \Ver_{G_i, N_i}: \prod_{i=1}^r G_i \longrightarrow \prod_{i=1}^r N_i^{\rm ab}.
  \end{equation}
\end{lemma}
\begin{proof}
Let $\pr_i: G\to G_i$ be the $i$-th projection. Then 
\[ \Ver_{G, \wt N_i/H_i}=\Ver_{G/H_i,\wt N_i/H_i} \circ \,\pr_i =\Ver_{G_i,N_i} \circ \,\pr_i,  \]
cf. Remark~\ref{rm:Hnormal}. Thus, the map \eqref{eq:prod_ver-1} is equal to $\prod_i \Ver_{G_i, N_i} \circ \, (\pr_i)_i$. Since the map $(\pr_i)_i:G\to \prod_i G_i$ is the identity, we show that the map \eqref{eq:prod_ver-1} is equal to $\prod_i \Ver_{G_i, N_i}$. \qed
\end{proof}

\subsection{Connection with class field theory}
\label{sec:T.5}

Let $k\subset E\subset K$ be three global fields. Let $L/k$ be a finite Galois extension containing $K$ with Galois group $G=\Gal(L/k)$.
Let $H=\Gal(L/K)\subset N=\Gal(L/E)$ be subgroups of $G$. 

The Artin map produces the following isomorphisms (cf.~\eqref{eq:cft})
\[ C_k/N_{L/k}(C_L)\simeq G^{\rm ab}, \quad C_k/N_{K/k}(C_K)\simeq (G/H)^{\rm ab}, \] \[ \quad C_E/N_{L/E}(C_L)\simeq N^{\rm ab}, \quad
  C_E/N_{K/E}(C_K)\simeq (N/H)^{\rm ab}. 
\]
We obtain the following commutative diagrams
\begin{equation}
  \label{eq:cft_ver}
  \begin{CD}
    C_k/N_{L/k}(C_L) @>>> C_k/N_{K/k}(C_K)\\
    @VVV @VVV \\
   C_E/N_{L/E}(C_L)  @>>>  C_E/N_{K/E}(C_K)\ 
  \end{CD} \qquad \qquad
   \begin{CD}
    G^{\rm ab}  @>\pi_H>> (G/H)^{\rm ab} \\
    @V{\Ver}VV @VV{\Ver_{G/H,N/H}}V \\
    N^{\rm ab} @>\pi_H>> (N/H)^{\rm ab}. \\ 
  \end{CD}
\end{equation}
From this we see that $\Ver_{G/H,N/H}: (G/H)^{\rm ab} \to (N/H)^{\rm ab}$ does not depend on the choice of the Galois extension $L/k$.

\section{Cohomology groups of algebraic tori}
\label{sec:C}

\subsection{$H^1(\Lambda^1)$ and $H^1(\Lambda)$} \label{sec:C.2}
Let $K=\prod_{i=1}^r K_i$ be a commutative etale $k$-algebra and $E=\prod_{i=1}^r E_i$ a $k$-subalgebra. Let $T^{K/E,k}$ and $T^{K/E,1}$ be the $k$-tori defined in Subsection~\ref{sec:P.3}. 
Let $L/k$ be a splitting Galois field extension for $T^K$, and let $G=\Gal(L/k)$. 
Put 
\[ \Lambda:=X(T^{K/E,k})\quad \text{and}\quad \Lambda^1:=X(T^{K/E,1}), \]
which are $G$-modules. Put $H_i:=\Gal(L/K_i)$, $\wt N_i:=\Gal(L/E_i)$ and $N_i^{\rm ab}=(\wt N_i/H_i)^{\rm ab}:=\wt N_i/H_i D(\wt N_i)$ for $1\le i\le
r$. In case that $K_i$ is Galois over $E_i$, we let $N_i:=\wt N_i/H_i=\Gal(K_i/E_i)$ and then $N_i^{\rm ab}$ is the abelianization of the group $N_i$.


From the diagram \eqref{eq:diagram_tori}, we obtain the commutative diagram of $G$-modules:
\begin{equation}\label{eq:diagram_char}
\begin{tikzcd}
\qquad 0 \arrow{r} & X(T^{E}) \arrow{r}{\wh N}\arrow{d}{\wh \Delta} 
& X(T^K) \arrow{d}{\wh j_T} \arrow{r}{\wh j}& \Lambda^1 \arrow{r} \arrow[equal]{d} & 0 \\
\qquad 0 \arrow{r} & \Z \arrow{r}{\wh N} & \Lambda \arrow{r}{\wh j}
 & \Lambda^1 \arrow{r} & 0.
\end{tikzcd}
\end{equation}
We have
\begin{equation}
  \label{eq:XTK}
   X(T^K)=
    \bigoplus_{i=1}^r \Ind_{H_i}^G \Z \quad \text{and} \quad X(T^E)=\bigoplus_{i=1}^r \Ind_{\wt N_i}^G \Z
\end{equation}
and  
\begin{equation}
  \label{eq:XKE1}
  T^{K/E,1}=\prod_{i=1}^r R_{E_i/k} R^{(1)}_{K_i/E_i} {\Gm}_{,K_i} \quad \text{and} \quad  \Lambda^1=\bigoplus_{i=1}^r \Ind_{\wt N_i}^G \Lambda^1_{E_i}, 
\end{equation}
where  $\Lambda^1_{E_i}:=X(R^{(1)}_{K_i/E_i} {\Gm}_{,K_i})$.
By Shapiro's lemma, one has 
\begin{equation}
    \label{eq:HqL1}
   H^q(G,\Lambda^1)=\bigoplus_{i=1}^r H^q(G, \Ind_{\wt N_i}^G \Lambda^1_{E_i})=\bigoplus_{i=1}^r H^q(\wt N_i,\Lambda^1_{E_i}), \quad \forall\, q\ge 0. 
\end{equation}
Note that if $K_i/E_i$ is Galois, then $H^1(\wt N_i,\Lambda^1_{E_i})=H^1(N_i,\Lambda^1_{E_i})$ (cf.~Remark~\ref{rem:indepXT}).

Since the torus $R^{(1)}_{K_i/E_i} {\Gm}_{,K_i}$ is anisotropic, one has $H^0(G,\Lambda^1)=\bigoplus_{i=1}^r H^0(\wt N_i,\Lambda^1_{E_i})=0$. 
From the exact sequence of algebraic $E_i$-tori
\begin{equation}
    \label{eq:R1Ei}
   \begin{CD}
    1 @>>> R^{(1)}_{K_i/E_i} (\G_{{\rm m}, K_i})
    @>j>>  R_{K_i/E_i} (\G_{{\rm m}, K_i}) @>N_{K_i/E_i}>> \G_{{\rm m}, E_i} \to 1,  
  \end{CD}     
\end{equation} 
one has an exact sequence of $\wt N_i$-modules
\begin{equation}
    \label{eq:L1Ei}
 \begin{CD}
  0 @>>> \Z @>\wh N >>  X(R_{K_i/E_i} (\G_{{\rm m}, K_i})) @>\wh j >> \Lambda^1_{E_i} @>>> 0.    
  \end{CD}    
\end{equation}
When $K_i/E_i$ is Galois, all $E_i$-tori in \eqref{eq:R1Ei} split over $K_i$ and hence \eqref{eq:L1Ei} becomes an exact sequence of $N_i$-modules and $X(R_{K_i/E_i} (\G_{{\rm m}, K_i}))\simeq \Ind^{N_i}_1 \Z$ is an induced module.  



Taking Galois cohomology to the the lower exact sequence of \eqref{eq:diagram_char}, we have an exact sequence
\begin{equation}\label{eq:H1L1d}
\begin{CD}
0 @>>> H^1(G,\Lambda) @>>> H^1(G,\Lambda^1) @>\delta>> H^2(G,\Z). 
\end{CD}
\end{equation}

\begin{prop}\label{prop:H1L}
  Let the notation be as above.

{\rm (1)} There is a canonical isomorphism $H^1(G,\Lambda^1)\simeq \bigoplus_{i} N_{i}^{\rm ab,\vee}$.  

{\rm (2)} Under the canonical isomorphisms $H^1(G,\Lambda^1)\simeq \bigoplus_{i} N_{i}^{\rm ab,\vee}$ and $H^2(G,\Z)\simeq G^{\rm ab, \vee}$~\eqref{eq:H2Z}, the map $\delta: H^1(G, \Lambda^1)\to H^2(G, \Z)$ 
expresses as
\begin{equation}
    \label{eq:delta}
    \sum_{i=1}^{r} \Ver_{G,N_i}^{\vee}\colon \bigoplus_{i=1}^{r} N_i^{\rm ab,\vee} \to G^{\rm ab,\vee},
\end{equation}
where $\Ver_{G,N_i}:G\to N_i^{\rm ab}$ is the transfer map.
In particular, $H^1(G,\Lambda)\simeq \Ker (\sum \Ver_{G,N_i}^{\vee})$.
Furthermore, the map in \eqref{eq:delta} factors as
\begin{equation} \label{eq:dec}
  \begin{CD}
  \bigoplus_{i=1}^{r} N_i^{\rm ab,\vee} @>{(\Ver_{G/H_i, N_i}^{\vee})_{i}}>>  &
  \bigoplus_{i=1}^{r} (G/H_i)^{\rm ab,\vee} @>{\sum_{i=1}^{r}\pi_{H_i}^\vee}>> &
   \ G^{\rm ab,\vee},  
  \end{CD}  
\end{equation}
where $\pi_{H_i}: G^{\rm ab} \to (G/H_i)^{\rm ab}$ the map mod $H_i$.

\end{prop}

\begin{proof}
(1) Taking Galois cohomology of the upper exact sequence of \eqref{eq:diagram_char}, we have a long exact sequence of abelian groups:
\begin{equation} \label{eq:1strow1}
\begin{tikzcd}
\qquad H^1(G, X(T^K)) \arrow{r} & H^1(G,\Lambda^1) \arrow{r}{\delta}  
& H^2(G, X(T^E)) \arrow{r}{\wh N^2}  &  H^2(G, X(T^K)).
\end{tikzcd}  
\end{equation}
By \eqref{eq:XTK}, the first term 
$H^1(G,X(T^K))=\bigoplus_i H^1(H_i,\Z)=0.$ 
Using the relations \eqref{eq:XTK} and by Shapiro's lemma, $H^1(G,\Lambda^1)=\bigoplus_i H^1(\wt N_i, \Lambda^1_{E_i})$ and the exact sequence \eqref{eq:1strow1} becomes 
\begin{equation} \label{eq:1row}
\begin{tikzcd}
\qquad 0 \arrow{r} & \bigoplus_{i=1}^{r} H^1(\wt N_i, \Lambda^1_{E_i}) \arrow{r}{\wt \delta}  
& \bigoplus_{i=1}^{r} H^2(\wt N_i,\Z) \arrow{r}{\rm Res}  & \bigoplus_{i=1}^{r} H^2(H_i,\Z).   
\end{tikzcd}
\end{equation}

It is clear that the following sequence 
\begin{equation} \label{eq:1row2}
\begin{tikzcd}
\qquad 0 \arrow{r} &  \Hom(N_i^{\rm ab}, \Q/\Z)  \arrow{r}{\rm Inf}  
&  \Hom(\wt N_i, \Q/\Z) \arrow{r}{\rm Res}  &  \Hom(H_i,\Q/\Z)
\end{tikzcd}
\end{equation}
is exact.
Using the canonical isomorphism 
$H^2(H,\Z)\simeq \Hom(H,\Q/\Z)$ \eqref{eq:H2Z} for any group $H$, 
we rewrite \eqref{eq:1row2} as follows:
\begin{equation} \label{eq:1row3}
\begin{tikzcd}
\qquad 0 \arrow{r} &  H^2(N_i^{\rm ab}, \Z)  \arrow{r}{\rm Inf}  
&  H^2(\wt N_i, \Z) \arrow{r}{\rm Res}  &  H^2(H_i,\Z).   
\end{tikzcd}
\end{equation}
Comparing \eqref{eq:1row3} and \eqref{eq:1row}, there is a unique isomorphism
\begin{equation}
  \label{eq:H1L1}
H^1(G,\Lambda^1)=\bigoplus_{i=1}^r H^1(\wt N_i, \Lambda^1_{E_i}) \simeq \bigoplus_{i=1}^r H^2(N_i^{\rm ab},\Z)   
\end{equation}
which fits the following commutative diagram
\begin{equation} \label{eq:1row4}
\begin{tikzcd}
\qquad 0 \arrow{r} &  \bigoplus_{i=1}^{r} H^2(N_i^{\rm ab}, \Z)  \arrow{r}{\rm Inf}  \arrow[equal]{d}
&  \bigoplus_{i=1}^{r} H^2(\wt N_i, \Z) \arrow{r}{\rm Res}  \arrow[equal]{d} &  \bigoplus_{i=1}^{r} H^2(H_i,\Z) \arrow[equal]{d} \\
\qquad 0 \arrow{r} & \bigoplus_{i=1}^{r} H^1(\wt N_i, \Lambda^1_{E_i}) \arrow{r}{\wt \delta}  
& \bigoplus_{i=1}^{r} H^2(\wt N_i,\Z) \arrow{r}{\rm Res}  & \bigoplus_{i=1}^{r} H^2(H_i,\Z).   
\end{tikzcd}
\end{equation}
This shows the first statement.

(2) The map $\wh \Delta$ in \eqref{eq:diagram_char} is induced from the restriction of $X(T^E)$ to the subtorus $\Gmk$ and therefore is given by
\begin{equation}
  \label{eq:hatD}
  \wh \Delta=\sum_{i=1}^{r} \wh \Delta_i \colon \bigoplus_{i=1}^r \Ind_{\wt N_i}^G \Z\to \Z, \quad \wh \Delta_i(g\otimes n)=n, \quad \forall\, g\in G, n\in\Z. 
\end{equation}
Thus, by the definition, the induced map $\wh \Delta_i^2$ on $H^2(G,-)$ is nothing but the corestriction $\Cor \colon H^2(\wt N_i, \Z) \to H^2(G,\Z)$; see Section \ref{sec:T.2}.  

From the diagram \eqref{eq:diagram_char}, we obtain the following commutative diagram:
\begin{equation} \label{eq:fund_diagram}
\begin{tikzcd}
\qquad  & 0 \arrow{r} \arrow{d} & \bigoplus_{i=1}^{r} H^1(\wt N_i, \Lambda^1_{E_i}) \arrow{r}{\wt \delta} \arrow[equal]{d} 
& \bigoplus_{i=1}^{r} H^2(\wt N_i,\Z) \arrow{r}{\rm Res} \arrow{d}{\wh \Delta^2=\Cor} & \bigoplus_{i=1}^{r} H^2(H_i,\Z) \arrow{d}    \\
\qquad 0 \arrow{r} & H^1(G,\Lambda) \arrow{r}  &  H^1(G,\Lambda^1)  \arrow{r}{\delta} & H^2(G,\Z)  \arrow{r}{}
 & H^2(G,\Lambda).  
\end{tikzcd}  
\end{equation}

With the identification \eqref{eq:H1L1}, we have 
$\delta= \Cor \circ \wt \delta= \Cor\circ \Inf $ and 
the lower long exact sequence of \eqref{eq:fund_diagram} becomes 
\begin{equation} \label{eq:4term_es}
\begin{tikzcd}
\qquad 0 \arrow{r} & H^1(G,\Lambda) \arrow{r}  & \bigoplus_{i=1}^{r} H^2(N_i^{\rm ab}, \Z) \arrow{r}{\Cor\circ \Inf} & H^2(G,\Z)  \arrow{r}{\wh N^2}
 & H^2(G,\Lambda).  & 
\end{tikzcd}  
\end{equation}
By Propositions~\ref{prop:rt} and \ref{prop:cor},
under the isomorphism~\eqref{eq:H2Z}
the map $\Cor \circ \Inf \colon H^2(N_i^{\rm ab},\Z) \to H^2(G,\Z)$ corresponds to 
$\Ver^{\vee}_{G,N_i}: N_i^{\rm ab,\vee} \to G^{\rm ab,\vee}$, where $\Ver^{\vee}_{G,N_i}$ is the dual of the transfer $\Ver_{G, N_i}: G^{\rm ab}\to N_i^{\rm ab}$ relative to $H_i$.
This proves \eqref{eq:delta}. Then it follows from  \eqref{eq:4term_es} that $H^1(\Lambda)\simeq \Ker (\sum_{i=1}^{r} \Ver^{\vee}_{G,N_i})$. As the map $\Ver_{G, N_i}: G^{\rm ab}\to N_i^{\rm ab}$ factors as 
\[ \begin{CD}
G @>\pi_{H_i}>> (G/H_i)^{\rm ab} @>\Ver_{G/H_i, \wt N_i/H_i}>> N_i^{\rm ab},
\end{CD}
\]
the last assertion \eqref{eq:dec} follows. \qed
\end{proof}

\begin{remark}\label{rem:indep}
  In terms of class field theory, we have $C_{E_i}/N_{K_i/E_i}(C_{K_i})\simeq N_i^{\rm ab}$, $C_k/N_{K_i/k}(C_{K_i})\simeq (G/H_i)^{\rm ab}$ and $C_k/N_{L/k}(C_L)\simeq G^{\rm ab}$. Let $L_0$ be the compositum of all Galois closures of $K_i$ over $k$; this is the minimal splitting field of the algebraic torus $T^{K/E,k}$. One has $L_0\subset L$ and its Galois group $G_0:=\Gal(L_0/k)$ is a quotient of $G$. Thus, we have the following commutative diagram:
  \begin{equation}
    \label{eq:indep}
    \begin{CD}
      \frac{C_k}{N_{L/k}(C_L)} @>>> \frac{ C_k}{N_{L_0/k}(C_{L_0})} @>>> \prod_i \frac{C_k}{N_{K_i/k}(C_{K_i})} @>>>   \prod_i \frac{C_{E_i}}{N_{K_i/E_i}(C_{K_i})} \\
      @VVV @VVV @VVV @VVV \\
      G^{\rm ab} @>>> G_0^{\rm ab} @>>> \prod_{i} (G/H_i)^{\rm ab} @>>> \prod_{i} N_i^{\rm ab}. 
    \end{CD}
  \end{equation}
  Taking the Pontryagin dual, the lower row gives
  \begin{equation}
    \label{eq:indp-dual}
    \begin{CD}
      \prod_{i} N_i^{\rm ab,\vee}  @>>> \prod_{i} (G/H_i)^{\rm ab,\vee} @>>> G_0^{\rm ab,\vee}
      \subset  G^{\rm ab,\vee}.
    \end{CD}
  \end{equation}
  From this, we see that the map $\sum \Ver_{G,N_i}^{\vee}$ in \eqref{eq:delta} has image contained in $G_0^{{\rm ab},\vee}$. It follows from \eqref{eq:indep} that this map is independent of the choice of the splitting field $L$, also cf. Remark~\ref{rem:indepXT}.   
\end{remark}

\subsection{$\Sha^2(\Lambda)$}
\label{sec:C.3}

Let $\scrD$ be the set of all decomposition groups of $G$.
For any $G$-module $A$, denote by
$$H^i_\scrD(A):=\prod_{D\in \scrD} H^i_D(A), \quad H^i_D(A):=H^i(D,A)$$ and
\[ r^i_{\scrD,A}=(r^i_{D,A})_{D\in \scrD}: H^i(G,A) \to H^i_\scrD(A) \]
the restriction map to subgroups $D$ in $\scrD$.
By definition, $\Sha^i(A)= \Ker r^i_{\scrD,A}$. 
We shall write $r_D$ and $r_{\scrD}$ for $r^i_{D,A}$ and $r^i_{\scrD,A}$, respectively, if it is clear from the content.

For the remainder of this section, we assume that the extension \emph{$K_i/E_i$ is cyclic} with Galois group $N_i$ for all $i$. Consider the following commutative diagram 
\begin{equation}
  \label{eq:lgd} 
\begin{CD}
  H^1(\Lambda^1) @>{\delta}>> H^2( \Z) @>{\wh N}>> H^2(\Lambda) @>{\wh j}>> H^2(\Lambda^1) \\
  @VV{r^1_{\scrD,\Lambda^1}}V @VV{r^2_{\scrD,\Z}}V @VV{r^2_{\scrD,\Lambda}}V @VV{r^2_{\scrD,\Lambda^1}}V \\
  H^1_{\scrD}(\Lambda^1) @>\delta_\scrD>>   H^2_{\scrD}( \Z) @>{\wh N}>> H^2_{\scrD}(\Lambda) @>{\wh j}>> H^2_{\scrD}(\Lambda^1). \\ 
\end{CD}
\end{equation}
Define
\begin{equation}
  \label{eq:H2Zp}
  \begin{split}
  H^2(\Z)':&=\{x\in  H^2(\Z) : \wh N (x) \in \Sha^2(\Lambda) \} \\ 
   &=\{x\in H^2(\Z) : r^2_{\scrD,\Z} (x) \in {\rm Im} (\delta_\scrD) \} .  
  \end{split}
\end{equation}

  

\begin{prop}
\label{prop:Sha2L-1}
  Assume that $K_i/E_i$ is cyclic with Galois group $N_i$ for all $i$. Then
  $\Sha^2(\Lambda^1)=0$, 
  $\Sha^2(\Lambda)\simeq H^2(\Z)'/{\rm Im}(\delta)$ and
  \begin{equation}
    \label{eq:tauT}
\tau_k(T^{K/E,k})=\frac{\prod_{i=1}^r |N_i|}{|H^2(\Z)'|}.    
  \end{equation}
\end{prop}
\begin{proof}
It is obvious that $\wh j(\Sha^2(\Lambda))\subset \Sha^2(\Lambda^1)$ and then we have a long exact sequence
\begin{equation}
  \label{eq:formula}
\begin{CD}
0 @>>> H^1(\Lambda) @>>>  H^1(\Lambda^1) @>{\delta}>> H^2( \Z)' @>{\wh N}>> \Sha^2(\Lambda) @>{\wh j}>> \Sha^2(\Lambda^1).
\end{CD}  
\end{equation}
  One has
  \[ \Sha^2(\Lambda^1)=\bigoplus_{i} \Sha^2(G,\Ind_{\wt N_i}^G \Lambda_{E_i}^1)
  =\bigoplus_{i} \Sha^2(\wt N_i, \Lambda_{E_i}^1)=\bigoplus_{i} \Sha^2(N_i, \Lambda_{E_i}^1)
  \]  
  because $\Sha^2(X(T))$ does not depend on the choice of the splitting field.
  Since $K_i/E_i$ is cyclic, by Chebotarev's
  density theorem, we have $\Sha^2(N_i, \Lambda^1_{E_i})=0$ for all $i$ and $\Sha^2(\Lambda^1)=0$.
Therefore, one gets the 4-term
  exact sequence
\begin{equation}
  \label{eq:formula-2}
\begin{CD}
0 @>>> H^1(\Lambda) @>>>  H^1(\Lambda^1) @>{\delta}>> H^2( \Z)' @>{\wh N}>> \Sha^2(\Lambda) @>{}>> 0.
\end{CD}  
\end{equation}
From this we obtain $\Sha^2(\Lambda)\simeq H^2(\Z)'/{\rm Im}(\delta)$ and
\[ \tau_k(T^{K/E,k})=\frac{|H^1(\Lambda)|}{|\Sha^2(\Lambda)|}=
  \frac{|H^1(\Lambda^1)|}{|H^2(\Z)'|}=\frac{\prod_{i=1}^r |N_i|}{|H^2(\Z)'|}. \quad \text{\qed}\]  
\end{proof}

\begin{remark}
  Ono showed \cite{ono:tamagawa} that $\tau(R^{(1)}_{K/k} \GmK)=[K:k]$ for any cyclic extension $K/k$. Since $|H^1(K/k, X(R^{(1)}_{K/k} \GmK))|=[K:k]$, it follows that $\Sha^2(K/k, X(R^{(1)}_{K/k} \GmK))=0.$ This gives an alternative proof of the first statement $\Sha^2(\Lambda^1)=0$ of Proposition~\ref{prop:Sha2L-1}. 
\end{remark}

In order to compute the groups $H^2(\Z)'$ and $\Sha^2(\Lambda)$, we describe the maps $\delta$ and $\delta_\scrD$ in the first commutative diagram of \eqref{eq:lgd}. 
Since $\Lambda^1=\oplus_i \Lambda^1_i$, where $\Lambda^1_i=X(T^{K_i/E_i,1})=\Ind^G_{\wt N_i} \Lambda^1_{E_i}$, it suffices to describe the following commutative diagram:
\begin{equation}
  \label{eq:lgd-1}
\begin{CD}
  H^1(\Lambda^1_i) @>{\delta}>> H^2( \Z) \\
  @VV{r_{D}}V @VV{r_{D}}V \\
  H^1_{D}(\Lambda^1_i) @>\delta_D>>   H^2_{D}( \Z).  \\ 
\end{CD}  
\end{equation}
Using the commutative diagram \eqref{eq:fund_diagram}, it factors as the following one:
\begin{equation}
  \label{eq:lgd-2}
\begin{CD}
   H^1(\Ind_{\wt N_i}^G \Lambda^1_{E_i}) @>{\wt \delta}>> H^2(\Ind_{\wt N_i}^G \Z)
  @>{\wh \Delta^2}>> H^2( \Z) \\
  @VV{r_{D}}V @VV{r_{D}}V @VV{r_{D}}V \\
   H^1_{D}(\Ind_{\wt N_i}^G \Lambda^1_{E_i}) @>\wt \delta_D>>   H^2_{D}(\Ind_{\wt N_i}^G \Z)@>{\wh \Delta^2_D} >>   H^2_{D}( \Z).  \\ 
\end{CD}  
\end{equation}
\begin{prop}\label{prop:rDdiagram}
  Assume that the extension $K_i/E_i$ is cyclic and both the subgroups $\wt N_i=\Gal(L/E_i)$ and $H_i=\Gal(L/K_i)$ are normal subgroups of $G$ (that is, $E_i/k$ and $K_i/k$ are Galois) for all $i$. There are natural isomorphisms $H^1(\Lambda^1)\simeq \bigoplus_iH^2(N_i,\Z)$ and $H^1_{\scrD}(\Lambda^1)=\bigoplus_{D} \bigoplus_{i}  H^2(\ol D_i,\Z)^{[G:D\wt N_i]}$, where $D_i=D\cap \wt N_i$ and $\ol D_i$ is its image in $N_i$. Under these identifications the first commutative diagram in \eqref{eq:lgd} decomposes as the following:
\begin{equation}
  \label{eq:lgd-7}
\begin{CD}
  \bigoplus_{i=1}^r H^2(N_i,\Z)  @>{\Inf}>> \bigoplus_{i=1}^r H^2({\wt N_i}, \Z)
  @>{\Cor}>> H^2( \Z) \\
  @VV{r_{\scrD}}V @VV{r_{\scrD}}V @VV{r_{\scrD}}V \\
  \bigoplus_{D} \bigoplus_{i=1}^rH^2(\ol D_i,\Z)^{[G:D\wt N_i]}  
  @>\Inf>>  \bigoplus_{D} \bigoplus_{i=1}^r H^2(D_i,\Z)^{[G:D\wt N_i]} @>{\Cor} >>   \bigoplus_{D} H^2_{D}( \Z).   
\end{CD}  
\end{equation}  
\end{prop}

\begin{proof}
For any element $g\in G$ and any $\wt N_i$-module $X$, let $X^g=X$ be the $g \wt N_ig^{-1}$-module defined by 
\[ h' \cdot x:= (g^{-1} h' g) x, \quad \text{for $x\in X^g$ and $h'\in g \wt N_ig^{-1}$}. \] 
Since $\wt N_i$ is normal in $G$, Mackey's formula \cite[Section 7.3 Prop.~22, p.~58]{Serre-linear} says that as $D$-modules one has
\begin{equation}
  \label{eq:M} 
  \Ind_{\wt N_i}^G X=\bigoplus_{g\in D\backslash G/\wt N_i} \Ind_{D_i}^D X^g,
\end{equation}
where $g$ runs through double coset representatives for $D\backslash G/\wt N_i$. 

By Mackay's formula \eqref{eq:M} and Shapiro's lemma, we obtain
the following commutative diagram from \eqref{eq:lgd-2}
\begin{equation}
  \label{eq:lgd-3}
\begin{CD}
   H^1(\wt N_i, \Lambda^1_{E_i}) @>{\wt \delta}>> H^2({\wt N_i}, \Z)
  @>{\Cor}>> H^2( \Z) \\
  @VV{r_{D}}V @VV{r_{D}}V @VV{r_{D}}V \\
  \bigoplus_{g} H^1(D_i, (\Lambda^{1}_{E_i})^g)
  @>\wt \delta_D>>  \bigoplus_{g} H^2(D, \Z)@>{\Cor} >>
  H^2_{D}( \Z).   
\end{CD}  
\end{equation}
We now show $(\Lambda^{1}_{E_i})^g\simeq \Lambda^{1}_{E_i}$.
Observe that $\Lambda_{E_i}=\Coker (\Z \to \Ind_{H_i}^{\wt N_i} \Z)$ and $\Ind_{H_i}^{\wt N_i} \Z=\Ind_1^{N_i} \Z$. It suffices to show that $(\Ind_1^{N_i} \Z)^g \simeq \Ind_1^{N_i} \Z$. Let $n^g=g^{-1} n g$ (for $n\in N_i$) be a $\Z$-basis of $(\Ind_1^{N_i} \Z)^g$. 
The new action is given by $h \cdot n_1^g:= h^g n_1^g=(hn_1)^g$ and one has $n_1^g n_2^g$, for $h, n_1, n_2\in N_i$. One sees that $(\Ind_1^{N_i} \Z)^g\simeq \Ind_1^{N_i} \Z$ and hence $(\Lambda^{1}_{E_i})^g\simeq \Lambda^{1}_{E_i}$.  
The bottom row of \eqref{eq:lgd-3} then expresses as
\begin{equation}
  \label{eq:lgd-4}
\begin{CD}
  H^1(D_i, \Lambda^{1}_{E_i})^{[G:D\wt N_i]}
  @>\wt \delta_D>>  H^2(D_i, \Z)^{[G:D \wt N_i]} @>{\Cor} >>   H^2_{D}( \Z).   
\end{CD}  
\end{equation}

Recall from \eqref{eq:L1Ei} we have an exact sequence of $\wt N_i$-modules
\begin{equation}
    \label{eq:L1Ei-2}
 \begin{CD}
  0 @>>> \Z @>\wh N >>  \Ind^{\wt N_i}_{H_i} \Z @>\wh j >> \Lambda^1_{E_i} @>>> 0    
  \end{CD}    
\end{equation}
on which $H_i$ acts trivially. 
Taking the Galois cohomology $H^{*}(D_i,-)$, we get an exact sequence
\begin{equation}
  \label{eq:lgd-5}
\begin{CD}
  0 @>>> H^1(D_i, \Lambda^{1}_{E_i})
  @>\wt \delta_{D_i}>>  H^2(D_i, \Z) @>\wh N >>   H^2(D_i, \Ind_{H_i}^{\wt N_i} \Z)\simeq H^2(D_i\cap H_i, \Z)^{[\wt N_i: D_i H_i]}   
\end{CD}  
\end{equation}
as one has
\[ H^2(D_i,\Ind_{H_i}^{\wt N_i} \Z)=H^2(D_i, \bigoplus_{\wt N_i/D_i H_i} \Ind_{D_i\cap H_i}^{D_i} \Z)\simeq H^2(D_i\cap H_i, \Z)^{[\wt N_i: D_i H_i]}. \]

Similar to \eqref{eq:1row3} and \eqref{eq:H1L1}, using the Inflation-Restriction exact sequence, we make the following identification $H^1(D_i, \Lambda^{1}_{E_i})=H^2(\ol D_i, \Z)$
and \eqref{eq:lgd-4} becomes
\begin{equation}
  \label{eq:lgd-6}
\begin{CD}
  H^2(\ol D_i, \Z)^{[G:D\wt N_i]}
  @>\Inf>>  H^2(D_i, \Z)^{[G:D \wt N_i]} @>{\Cor} >>   H^2_{D}( \Z).   
\end{CD}  
\end{equation}
This proves the proposition. \qed
\end{proof}

The following proposition gives a group-theoretic description of $\Sha^2(\Lambda)$.

\begin{prop}\label{prop:Sha2L}
  Let the notation and assumptions be as in Proposition~\ref{prop:rDdiagram}.
  Let
  \[ \Ver_{G, \bfN}=(\Ver_{G, N_i})_i : G^{\rm ab} \to \prod_i N_i \text{ and }
     \Ver_{D, \ol \bfD}=(\Ver_{D, \ol D_i})_i: D^{\rm ab}\to \prod_{i} \ol D_i \]  
   denote the corresponding transfer maps, respectively. Then
\begin{equation}
  \label{eq:H2Zp-1}
  H^2(G,\Z)'=\{f \in G^{{\rm ab}, \vee}: f {\big |}_{D^{\rm ab}}
  \in {\rm Im}(\Ver_{D, \ol \bfD}^\vee)\  \forall\, D\in \scrD \},  
\end{equation}
and
\begin{equation}
  \label{eq:Sha2L-gt} 
  \Sha^2(\Lambda) \simeq \frac{\{f \in G^{{\rm ab}, \vee}: f {\big |}_{D^{\rm ab}} \in {\rm Im}(\Ver_{D, \ol \bfD}^\vee) \ \forall\, D\in \scrD \}}{{\rm Im} (\Ver^{\vee}_{G,\bfN})}.  
\end{equation}
\end{prop}
\begin{proof}
  We translate the commutative diagram \eqref{eq:lgd-7} 
  in terms of group theory. For each decomposition group $D\in \scrD$, we have the following corresponding  commutative diagram:
\begin{equation}
  \label{eq:lgd-8}
\begin{CD}
  \prod_{i=1}^r N_i  @<{\pi_{\bf H}=(\pi_{H_i})}<< \prod_{i=1}^r \wt N_i^{\rm ab} 
  @<{\Ver_{G,\bfN}=(\Ver_{G,N_i})}<< G^{\rm ab}\\
  @AA{\rm func}A @AA{\rm func}A @AA{\rm func}A \\
  \prod_{i=1}^r \ol D_i  @<{\pi_{{\bf D \cap  H}}=(\pi_{D_i\cap H_i})}<<  \prod_{i=1}^r D_i^{\rm ab}  @<{\Ver_{D,\bfD}=(\Ver_{D,D_i})}<<  D^{\rm ab}.   
\end{CD}  
\end{equation}  
 Here we ignore the multiplicity $[G:D\wt N_i]$ because we are only concerned with the image of the map $\Cor\circ \Inf$ in \eqref{eq:lgd-7} and this does not effect the result. 
Then $\Ver_{G, \bfN}: G^{\rm ab} \to \prod_i N_i$ and $\Ver_{D, \ol \bfD}: D^{\rm ab}\to \prod_{i} \ol D_i$ are the respective compositions. By Proposition~\ref{prop:rDdiagram}, the map $\delta_D$ corresponds to
$\Ver_{D,\ol \bfD}^{\vee}$. From the second description of $H^2(\Z)'$ in  \eqref{eq:H2Zp}, we obtain \eqref{eq:H2Zp-1}. 

By Propositions~\ref{prop:rDdiagram} and \ref{prop:H1L}, the map $\delta: H^1(\Lambda^1) \to H^2(\Z)'\subset H^2(\Z)$ \eqref{eq:formula-2} corresponds to $\Ver_{G, \bfN}^\vee: \prod_i N_i^\vee \to G^{{\rm ab},\vee}$.    
Thus, by Proposition~\ref{prop:Sha2L-1}, we obtain \eqref{eq:Sha2L-gt}. This proves the proposition. \qed
\end{proof}

\section{Computations of some product cases}
\label{sec:MC}
We keep the notation in the previous section. In this section we consider the case where the extensions $K_i/k$ are all Galois with Galois group $G_i=\Gal(K_i/k)$. Let $L=K_1K_2\cdots K_r$ be the compositum of all $K_i$ over $k$ with Galois group $G=\Gal(L/k)$. Assume that 
\begin{itemize}
    \item[(i)]  the canonical map monomorphism $G\to G_1\times \dots \times G_r$  is an isomorphism, and
    \item[(ii)] $K_i/E_i$ is cyclic with Galois group $N_i:=\Gal(K_i/E_i)$ for all $i$.
\end{itemize}
As before, we put $T=T^{K/E,k}$, $\Lambda:=X(T)$, $T^1=T^{K/E,1}$ and $\Lambda^1:=X(T^{1})$.

\subsection{$H^1(\Lambda)$ and $H^2(\Lambda^1)$} \label{sec:MC.1}

\begin{lemma}\label{lm:H1L_prod}
  Let $\Lambda_i:=X(T^{K_i/E_i,k})$. We have $H^1(\Lambda)\simeq \oplus_{i=1}^r H^1(\Lambda_i)$.
\end{lemma}
\begin{proof}
By Proposition~\ref{prop:H1L}, $H^1(\Lambda)$ is isomorphic to the kernel of the dual of the map $\prod_i \Ver_{G,N_i} : G \to \prod_i N_i^{\rm ab}$. By Lemma~\ref{lm:prod_ver}, $\prod_i \Ver_{G,N_i}=\prod_i \Ver_{G_i, N_i}: \prod_{i=1}^r G_i \to \prod_i N_i^{\rm ab}$, 
and therefore the kernel of its dual is equal to the product of abelian groups $H^1(\Lambda_i)$ for $i=1,\dots, r$, by Proposition~\ref{prop:H1L}. This proves the lemma. \qed
\end{proof}

We remark that Condition (ii) is not needed in the proof of Lemma~\ref{lm:H1L_prod}. 

\begin{prop}\label{prop:HS5n}
  Let $H\subset G$ be a normal subgroup of a finite group $G$, let $A$ be a $G$-module, and let $n\ge 1$ be a positive integer. If $H^q(H, A)=0$ for all $0 < q < n$, then we have a 5-term long exact sequence
\begin{equation} \label{eq:HS5n}
0 \to H^n(G/H, A^H) \to H^n(G, A) \to H^n(H, A)^{G/H} \xrightarrow{d} H^{n+1}(G/H, A^H) \to H^{n+1}(G, A). 
\end{equation}
\end{prop}
\begin{proof}
This follows from the Hochschild-Serre spectral sequence
\[ E_2^{p, q}:=H^p(G/H, H^q(H, A^H)) \implies H^{p+q}(G,A) \]
(cf.~\cite[Theorem 2.1.5, p.~82]{NSW}) and \cite[Proposition 2.1.3, p.~81]{NSW}. \qed

\end{proof}

\begin{prop}\label{prop:H2L1-prod}
  Let $K=\prod_{i=1}^r K_i$ and $E=\prod_{i=1}^r E_i$ be as before. Suppose that each $K_i/k$ is Galois with group $G_i$, and that Conditions (i) and (ii) are satisfied. 
  Then $H^2(G, \Lambda^1)\simeq \oplus_{i=1}^r H^2(\wt N_i, \Lambda_{E_i}^1)$ and there is a natural homomorphism 
  $v_i: N_i \to (H_i^{\rm ab})^{|N_i|-1}$ such that 
  $H^2(\wt N_i, \Lambda_{E_i}^1)\simeq \Ker(v_i^\vee)$. 
If $r=1$, that is $K/k$ is a Galois field extension, then $H^2(\Lambda^1)=0$.
\end{prop}

\begin{proof}
By \eqref{eq:HqL1}, we have 
\begin{equation}
    \label{eq:HqL1-1}
   H^q(G,\Lambda^1)=\bigoplus_{i=1}^r H^q(\wt N_i,\Lambda^1_{E_i}), \quad \forall\, q\ge 0. 
\end{equation}
By \eqref{eq:L1Ei-2}, we get the long exact sequence
\begin{equation}
    \label{eq:HqNiL1}
  \begin{CD}
    H^q(N_i, \Ind_{H_i}^{\wt N_i} \Z) @>>> 
    H^q(N_i, \Lambda^1_{E_i}) @>>>
    H^{q+1}(N_i,  \Z) @>>>
    H^{q+1}(N_i, \Ind_{H_i}^{\wt N_i} \Z). 
  \end{CD}
\end{equation}
Since $\Ind_{H_i}^{\wt N_i} \Z=\Ind_1^{N_i} \Z$  is an induced module, 
it follows from \eqref{eq:HqNiL1} that 
\begin{equation}\label{eq:H2L1Ei}
H^2(N_i, \Lambda^1_{E_i})\isoto H^3(N_i, \Z)=H^1(N_i, \Z)=0.     
\end{equation}
Since the algebraic torus $T^{K_i}$ splits over $K_i$, the $H_i$-module $\Lambda^1_{E_i}\simeq \Z^{{|N_i|-1}}$ is trivial and $H^1(H_i, \Lambda^1_{E_i})=H^1(H_i, \Z^{{|N_i|-1}})=0$.
By Proposition~\ref{prop:HS5n} with $(G,H)=(\wt N_i, H_i)$, 
we have the exact sequence
\begin{equation}
    \begin{CD}
    0   @>>> H^2(\wt N_i, \Lambda^1_{E_i}) @>>> H^2(H_i, \Lambda^1_{E_i})^{N_i} @>{d_i}>> H^3(N_i, \Lambda^1_{E_i}).
    \end{CD}
\end{equation}
Using the same argument as \eqref{eq:H2L1Ei}, we get $H^3(N_i,\Lambda^1_{E_i})\simeq \Hom(N_i,\Q/\Z)$.
On the other hand, 
\begin{equation}\label{eq:H2HiL1Ei}
  H^2(H_i, \Lambda^1_{E_i})\simeq H^2(H_i, \Z)^{{|N_i|-1}}\simeq \Hom(H_i,\Q/\Z)^{{|N_i|-1}}.   
\end{equation}

The group $N_i=\wt N_i/H$ acts on  $\Hom(H_i,\Q/\Z)^{{|N_i|-1}}$ by conjugation and the action is trivial. 
Thus, we obtain an exact sequence
\begin{equation}\label{eq:di}
   \begin{CD}
    0   @>>> H^2(\wt N_i, \Lambda^1_{E_i}) @>>> \Hom((H_i^{\rm ab})^{{|N_i|-1}}, \Q/\Z) @>d_i>> \Hom(N_i, \Q/\Z).
    \end{CD}    
\end{equation}
Let $v_i:N_i \to (H_i^{\rm ab})^{{|N_i|-1}}$ be the Pontryagin dual of $d_i$ and then $H^2(\wt N_i, \Lambda^1_{E_i})=\Ker (v_i^\vee)$. 
\qed 

\end{proof}

\begin{cor}\label{cor:H2L1-prod}
  Let the notation and assumptions be as in Proposition~\ref{prop:H2L1-prod}. Assume further that the orders of groups $G_i$ are mutually relatively prime.  Then 
 \begin{equation}\label{eq:H2L1-abel}
    H^2(\Lambda^1)\simeq \bigoplus_{i=1}^r \Hom(H_i, \Q/\Z)^{[K_i:E_i]-1}. 
  \end{equation}
(In our notation the group $H_1=\{1\}$ if $r=1$.)
\end{cor}

\begin{proof}
In this case, the maps $d_i$'s are all zero. \qed
\end{proof}

\begin{rem} It is curious to know what the natural homomorphism $v_i:N_i \to (H_i^{\rm ab})^{[K_i/E_i]-1}$ is. 
\end{rem}

\subsection{$\tau(T)$}\label{subsec:tauT}
In this subsection we shall further assume that each subgroup $N_i$, besides being cyclic, is also normal in $G_i$. This assumption simplifies the description of the group $H^2(\Z)'$ through Mackey's formula. 
For each
$1 \le i \le r$ 
let $T_i:=T^{K_i/E_i, k}$, and let $\Lambda_i:=X(T_i)$ be the character group of $T_i$. Let $\scrD_i$ be the decomposition subgroups of $G_i$ for each $i$.

\begin{lemma}\label{lm:(a)(b)}
{\rm (1)} The inclusion $\prod_{i=1}^r H^2(G_i,\Z)'\subset H^2(\Z)'$ holds. 

{\rm (2)} Assume that for any $1\le i \le r$ and $D_i'\in \scrD_i$, there exists a member $D\in \scrD$ such that $\pr_i(D)=D'_i$ and a section $s_i: D'_i\to D$ of $\pr_i:D\twoheadrightarrow D_i'$ such that 
the composition $\pr_k \circ s_i\colon D_i'\to D \to G_{k}$ is trivial for $k\neq i$. Then $\prod_{i=1}^r H^2(G_i,\Z)'= H^2(\Z)'$.
\end{lemma}
\begin{proof}
(1) For each $f=(f_1, \dots, f_r)\in \Hom(G, \Q/\Z)=\prod_{i=1}^r \Hom(G_i, \Q/\Z)$,
consider the following two conditions:


\begin{enumerate}
    \item[\rm (a)] For each $D\in \scrD$, the restriction $f|_{D_{}^{}}$ of $f$ to $D^{}$ lies in the image of the map $\sum_{i=1}^r \Ver_{D,\ol D_i}^\vee$.
    
    \item[\rm (b)] For each $1 \le i\le r$ and $D'_i\in \scrD_i$, 
    the restriction $f_i|_{D_{i}'^{}}$ to $D_{i}'^{}$ lies in the image of the map 
    $\Ver_{D_{i}', D_{i}' \cap N_i}^\vee$.
\end{enumerate}
It suffices to show that the condition (b) implies (a).

Let $D=D_\grP\in \scrD$ be a decomposition group associated to a prime $\grP$ of $L$.
By definition $D_i=D\cap \wt N_i$ and $\ol D_i=D_i/D_i \cap H_i$. The projection map $\pr_i: G\to G_i$ sends each element $\sigma$ to its restriction $\sigma|_{K_i}$ to $K_i$. Let $\grP_i$ denote the prime of $K_i$ below $\grP$ and we have $\pr_i(D_{\grP})= D_{\grP_i}$.  
We show that $\pr_i(D_\grP \cap \wt N_i)=D_{\grP_i} \cap N_i$. The inclusion $\subseteq$ is obvious because $\pr_i(D_\grP \cap \wt N_i)\subset \pr_i(D_{\grP_i}) \cap \pr_i(\wt N_i)=D_{\grP_i} \cap N_i$. For the other inclusion, since $\pr_i: D_\grP\to D_{\grP_i}$ is surjective, for each $y\in D_{\grP_i} \cap N_i$ there exists an element $x=(x_i)\in D_\grP$ such that $x_i=y$. Since $x_i\in N_i$, $x$ also lies in $\wt N_i$. This proves the other inclusion. Therefore, $\ol D_i=D_{\grP_i} \cap N_i$. 

The map $\sum_{i=1}^r \Ver_{D,\ol D_i}^\vee$ is dual to the map
\[ \prod_{i=1}^r \Ver_{D, D_{\grP_i} \cap N_i}: D \to \prod_{i=1}^r (D_{\grP_i} \cap N_i). \]
Since each $\Ver_{D, D_{\grP_i} \cap N_i}= \Ver_{D_{\grP_i}, D_{\grP_i} \cap N_i}\circ \, \pr_i$ (cf.~\eqref{eq:rV}), 
the above map factorizes into the following composition
\[ 
\begin{CD}
D @>(\pr_i)_i>> \prod_{i=1}^r D_{\grP_i} @>
\prod_{i} {\Ver_{D_{\grP_i}, D_{\grP_i} \cap N_i}}>> \prod_{i=1}^r (D_{\grP_i} \cap N_i).
\end{CD}
\]
Consider the restriction $f|_{\prod_i D_{\grP_i}}=(f_i |_{D_{\grP_i}})$ of $f$ to $\prod_i D_{\grP_i}$, and put $D_i':=D_{\grP_i}$. 
By condition (b), there exists an element $h_i\in \Hom(D_{\grP_i} \cap N_i, \Q/\Z)$ such that $f_i |_{D_{\grP_i}} = \Ver^\vee_{D_{\grP_i}, D_{\grP_i} \cap N_i} (h_i)$. Then $f|_{\prod_i D_{\grP_i}}=\left (\Ver^\vee_{D_{\grP_i}, D_{\grP_i} \cap N_i} (h_i) \right )$. 
Restrict this function to $D$, we obtain $f|_D=\sum_{i=1}^r \Ver_{D,\ol D_i}^\vee (h_i)$. This shows that $f|_D$ lies in the image of $\sum_{i=1}^r \Ver_{D,\ol D_i}^\vee$. 

(2) By (1), we have $\prod_{i=1}^r H^2(G_i,\Z)'\subset H^2(\Z)'$. To prove the other inclusion, we must show that each $f_i$ satisfies condition (b)  provided $f$ satisfies condition (a), which we assume from now on. For $1\leq i\leq r$, let $D\in \scrD$ be a decomposition group of $G$ over $D_i'$ and $s_i:D_i'\to D$ be a section such that $\pr_k \circ s_i$ is trivial for $k\neq i$. Then (a) implies
\[ f|_D=\sum_{k=1}^r \Ver_{D,\ol D_k}^\vee (h_k) \] for some $h_k\in \Hom(\ol D_k, \Q/\Z)$. Pulling back to $D_i'$ via $s_i$, we have 
\[ f_i|_{D_i'}=s_i^* f|_D=\sum_{k=1}^r s_i^* \Ver_{D,\ol D_k}^\vee (h_k)=\sum_{k=1}^r s_i^* \pr_k^* \Ver_{D'_k,\ol D_k}^\vee (h_k), \]
where $D'_{k}:=\pr_{k}(D)$ if $k\neq i$. Since $s_i^*\pr_k^*=1$ if $k=i$ and  $s_i^*\pr_k^*=0$ otherwise, we get $f_i|_{D_i'}=\Ver_{D'_i,\ol D_i}^\vee (h_i)$. \qed

\end{proof}

\begin{prop}\label{prop:product-cyclic-dec}
With the notation and assumptions be as above. Suppose further that for each $1\le i\le r$, every decomposition group of $G_i$ is cyclic. Then we have $\tau(T)=\prod_{i=1}^r \tau(T_i)$.
\end{prop} 
\begin{proof}
We shall show that the assumption in Lemma~\ref{lm:(a)(b)} (2) holds. Let $D'_i\in \scrD_i$ be a decomposition group of $G_i$; by our assumption $D'_i$ is cyclic and let $a_i$ be a generator.
Let $D$ be a cyclic subgroup of $G$ generated by $\tilde a_i=(1, \dots, 1,a_i, 1,\dots, 1)$ with $a_i$ at the $i$-th place.  
Clearly $D$ is the decomposition group of some prime and let $s_i:D_i' \to D$ be the section sending $a_i$ to $\tilde a_i$. Clearly for $k\neq i$, $\pr_k \circ s_i:D_i' \to D_k'$ is trivial. By Lemma~\ref{lm:(a)(b)}, we have $H^2(G,\Z)'=\prod_{i=1}^r H^2(G_i,\Z)'$. 

By Proposition~\ref{prop:Sha2L-1}, we have
\begin{equation}\label{eq:tauTi}
    \tau(T)=\frac{\prod_{i=1}^r |N_i|}{|H^2(\Z)'|} \quad \text{and} \quad \tau(T_i)=\frac{|N_i|}{|H^2(G_i,\Z)'|}.
\end{equation}
Therefore, $\tau(T)=\prod_{i=1}^r \tau(T_i)$. \qed
\end{proof}

\section{Galois CM fields}
\label{sec:SG} 
We return to the case of CM tori and keep the notation of Section~\ref{sec:P.3}.
In this section, we assume that $K$ is a Galois CM field. Let $G=\Gal(K/\Q)$, $G^+:=\Gal(K^+/\Q)$ and $g:=[K^+:\Q]$. 
Then one has the short exact sequence
\[1 \rightarrow \langle\iota\rangle \rightarrow G \rightarrow G^{+} \rightarrow 1. \tag{$*$}\]
Let $T=T^{K,\Q}$ be the associated CM torus over $\Q$.


\subsection{Cyclotomic extensions}

\begin{prop} 
Let $K=\mathbb{Q}(\zeta_{n})\neq \Q$ be the $n$th cyclotomic field with $n$ odd or $4|n$, and $T$ the associated CM torus over $\Q$. 

\begin{enumerate} 
    \item[$(1)$] If $n$ is either a power of an odd prime $p$ or $n=4$, then $\tau(T) = 1.$
    \item[$(2)$] In other cases, $\tau(T) = 2$.
\end{enumerate}
\end{prop}
\begin{proof} 
{(1)}  In this case the Galois group $G=\Gal(K/\Q)$ is cyclic and Proposition~\ref{LiRu}(3) shows that $\tau(T) = 1$. 

$(2)$ Suppose first that $n$ is not a power of a prime. By Proposition~\ref{formula}(2) the global norm index 
  $n_{K}$ divides $\prod_{p \in S_{K/K^{+}}}e_{T,p}$. 
  By \cite[Proposition~2.15]{Washington-cyclotomic}, the quadratic extension $K/K^{+}$ is unramified at all finite places of $K^+$. It follows that $S_{K/K^{+}} = \emptyset$, and hence that $n_{K} = 1$.
Proposition~\ref{formula}(1) then shows $\tau(T)=2$. 
    Now suppose the other case that
  $n=2^v$ with $v\ge 3$. Then $K$ contains $\mathbb{Q}(\zeta_{8})=\Q(\sqrt{2},\sqrt{-1})$,
  which contains two distinct imaginary quadratic fields. Thus,
  by Lemma~\ref{lm:E}, we have $n_{K} = 1$ and hence $\tau(T)=2$ by Proposition~\ref{formula}(1). 
  This completes the proof. \qed
\end{proof}

\subsection{Abelian extensions}

\begin{prop}\label{prop:ab-1}
Assume that $G=\Gal(K/\Q)$ is abelian.
\begin{enumerate}
    \item[$(1)$] We have $\tau(T) \in \{1,2\}$.
    \item[$(2)$] If $g$ is odd, then $\tau(T) = 1$.
    \item[$(3)$] If $g$ is even and $(*)$ splits, then $\tau(T) = 2$.
\end{enumerate}
\end{prop}
\begin{proof} 
(1) Write $G = \prod_{i=1}^{\ell} G_{i}$ as a product of cyclic subgroups $G_i$.
Then there exists $i$ such that the image of $\iota$ in the projection 
$G \to G_{i}$ is nontrivial, say $i=1$. Let $K_{1}$ be a subfield of $K$ with $\Gal(K_{1}/\mathbb{Q}) = G_{1}$, which is a cyclic CM field.
Then we have $\tau(T^{K_{1}, \mathbb{Q}})=1$ by Proposition~\ref{LiRu}(3), that is, $n_{K_1} = 2$. Since $K\supset K_1$, we have $n_{K} \mid n_{K_{1}}$ and $n_{K} \in \{1,2\}$. Therefore, $\tau(T) \in \{1,2\}$. 

(2) This follows from Proposition~\ref{LiRu}. 

(3) Since $(*)$ splits, $G \simeq \langle\iota\rangle \times G^{+}$. As $g$ is even, there is an epimorphism 
\begin{equation}\label{eq:biquad}
\langle\iota\rangle \times G^{+} \twoheadrightarrow
\langle\iota\rangle \times \mathbb{Z}/2\mathbb{Z}.   
\end{equation}
In particular, $K$ contains two distinct imaginary quadratic fields; thereby
$n_K=1$ and $\tau(T)=2$ by Lemma~\ref{lm:E}. \qed

\end{proof}

\subsection{Certain Galois extensions}

\begin{prop}
Assume the short exact sequence $(*)$ splits. Let $g^{ab} := |G^{+ {\rm ab}}|$ and $\Lambda:=X(T)$. Then $\tau(T)\in \{1,2\}$. Moreover, the following statements hold:

{\rm (1)} When $g$ is odd, $\tau(T) = 1$.

{\rm (2)} When $g$ is even,
\begin{enumerate} 
  \item[{\rm (i)}] if $g^{ab}$ is even, then $\tau(T) = 2$; 
  \item[{\rm (ii)}] if $g^{ab}$ is odd, then there is a unique nonzero
          element $\xi$ in the $2$-torsion subgroup $H^{2}(\Lambda)[2]$ 
          of $H^{2}(\Lambda)$. Moreover, $\tau(T) = 1$ if and only if
          its restriction $r_D(\xi) = 0$ in $H^2(D,\Lambda)$ for all $D\in \scrD$. 
\end{enumerate}
\end{prop} 

\begin{proof} 
    Since $(*)$ splits, $G \simeq \langle\iota\rangle \times G^{+}$. It follows that $K$ contains an imaginary quadratic field $E$. Since $n_E=2$, one has $n_K\in\{1,2\}$ and $\tau(T)\in \{1,2\}$.
   The statement (1) follows from Proposition~\ref{LiRu}(2). 
 
    (2) We have $G \simeq \langle\iota\rangle \times G^{+}$ and $G^{\rm ab}\simeq \<\iota\>\times G^{+ {\rm ab}}$.
    {\rm (i)} Suppose $g^{ab}$ is even. As \eqref{eq:biquad}, $K$ contains
          two distinct imaginary quadratic fields. Thus, $n_{K} = 1$ and $\tau(T) = 2$. 
{\rm (ii)} Suppose $g^{ab}$ is odd. Put $\Lambda^1=X(T^{K,1})$ and we have the following exact sequence
$$0 \rightarrow
      H^{1}(\Lambda) \rightarrow H^{1}(\Lambda^1) \rightarrow
      H^{2}(\mathbb{Z}) \rightarrow H^{2}(\Lambda) \rightarrow H^2(\Lambda^1). $$
      We have $H^{1}(\Lambda^1)\simeq \Z/2\Z$ by Proposition~\ref{prop:H1L}(1), $H^2(\Lambda^1)=0$ by Proposition~\ref{prop:H2L1-prod}, and $H^{1}(\Lambda)\simeq \Z/2\Z$ by \cite[Proposition A.6]{achter-altug-gordon}. 
This gives $H^{2}(\mathbb{Z}) \cong H^{2}(\Lambda).$ 
The proof of \cite[Lemma A.7]{achter-altug-gordon} shows that $H^1(G,T)$ is a $2$-torsion group. It follows that $\Sha^2(\Lambda)$ is also a $2$-torsion group. Therefore, $\Sha^2(\Lambda)\subset H^2(\Lambda)[2]$. 
Since $g^{\rm ab}$ is odd, 
$$ H^2(\Lambda)[2]\simeq H^2(\Z)[2]=\Hom(G^{\rm ab}, \Q/\Z)[2]=\Hom(\langle\iota\rangle \times {G^{+{\rm ab}}},
\mathbb{Q}/\mathbb{Z})[2]\cong
\mathbb{Z}/2\mathbb{Z}.$$

Let $\xi \in H^{2}(\Lambda)[2]$ be the unique nonzero
element. Then 
\begin{align*}
\tau(T)=1 & \iff \Sha^2(\Lambda)\simeq \Z/2\Z \iff \xi \in \Sha^{2}(\Lambda) \\
& \iff r_D(\xi)=0 \quad \text{for all $D\in \scrD$. } 
\end{align*}
This completes the proof of the proposition. \qed
\end{proof}

\begin{rmk}
When $G$ is non-abelian and the short exact sequence $(*)$ does not split, 
we do not know the value of $\tau(T)$ in general. However, 
if the involution $\iota$ is non-trivial on the maximal abelian extension 
$K_{ab}$ over $\Q$ in $K$, 
then $K_{ab}$ is a CM abelian field and it follows from Proposition~\ref{prop:ab-1} that $\tau(T) \in \{1,2\}$.

It remains to determine the Tamagawa number when $G$ is non-abelian, $(*)$ is non-split, and $K_{ab}$ is totally real. This includes the cases of $G=Q_{8}$ and the dihedral groups, for which we treat in the next subsections.
\end{rmk}


\subsection{$Q_8$-extensions}

The quaternion group $Q_8=\{\pm 1, \pm i, \pm j, \pm k\}$ is the group of $8$ elements generated by $i, j$ with usual relations $i^2=j^2=-1$ and $ij=-ji=k$. 
We have the following well-known properties of $Q_8$. 
\begin{lemma}\label{lm:Q8} 
\begin{itemize}
    \item[(1)] The group $Q_8$ contains $6$ elements of order $4$, one element of order $2$, and the identity. It does not has a subgroup isomorphic to $\Z/2\Z\times \Z/2\Z$. Thus, every proper subgroup is cyclic.
    \item[(2)] Every non-trivial subgroup contains $\{\pm 1\}$, and only subgroup which maps onto $Q_8/\{\pm 1\}$ is $Q_8$.
    \item[(3)] The center $Z(Q_8)=\{\pm 1\}=D(Q_8)$, where $D(Q_8)=[Q_8,Q_8]$.
\end{itemize}
\end{lemma}

\begin{prop}\label{prop:Q8CM}
Let $P$ and $Q$ be two odd positive integers such that 
\begin{equation}\label{eq:PQ}
  P-1=a^2,\quad Q-1=Pb^2  
\end{equation}
for some integers $a,b\in  \bbN$. Assume that $Q$ is not a square. Let $K:=\Q(\beta)$ be the simple extension of $\Q$ generated by $\beta$ which satisfies $\beta^2=\alpha:=-(P+\sqrt{P})(Q+\sqrt{Q})$. Then  
$K$ is a  Galois CM field with Galois group $Q_8$ with maximal totally real field $K^+=\Q(\sqrt{P},\sqrt{Q})$.  The Galois group $\Gal(K/\Q)$ is generated by $\tau_1$ and $\tau_2$ given by  
\begin{equation}
    \label{eq:tau12}
   \tau_1(\beta)=\frac{(\sqrt{P}-1)}{a} \beta, \quad  \tau_2(\beta)=\frac{(\sqrt{Q}-1)}{\sqrt{P} a} \beta. 
\end{equation}
Moreover, for each prime $\ell$, the decomposition group $D_\ell$ is cyclic except when $\ell|Q$. For $\ell|Q$, one has
\begin{equation}
  \label{eq:DlQ}
  D_\ell=\begin{cases}
  \Z/4\Z, & \text{if $\left(\frac{P}{\ell} \right)=1$;} \\
  Q_8 ,  & \text{if $\left(\frac{P}{\ell} \right)=-1$.} 
  \end{cases}
\end{equation}
\end{prop}

\begin{proof}
Note that $\gcd(P,Q)=1$ and $P$ is not a square. Then $\Q(\sqrt{P},\sqrt{Q})$ is a totally real  biquadratic field and its Galois group is generated by $\sigma_1$ and $\sigma_2$, where 
\[ \sigma_1: \begin{cases} \sqrt{P} \mapsto -\sqrt{P} \\
\sqrt{Q} \mapsto \sqrt{Q}\end{cases}, \quad 
\sigma_2: \begin{cases} \sqrt{P} \mapsto \sqrt{P} \\
\sqrt{Q} \mapsto -\sqrt{Q}. \end{cases} \]

It is clear that $\alpha$ is a totally negative, and hence $K$ is a CM field. The maximal totally real subfield $K^+=\Q(\alpha)\subset \Q(\sqrt{P},\sqrt{Q})$.  
Since the minimal polynomial of $\alpha$ is of degree $4$, one has $K^+=\Q(\sqrt{P},\sqrt{Q})$. 
Let $\tau_i\in \Gal(\Qbar/\Q)$, $i=1,2$, be elements such that $\tau_i|_{K^+}=\sigma_i$. One computes
\[ \frac{\tau_1(\beta)^2}{\beta^2}= \frac{\tau_1(\alpha)}{\alpha}=\frac{-\sigma_1(P+\sqrt{P})\sigma_1(Q+\sqrt{Q})}{-(P+\sqrt{P})(Q+\sqrt{Q})}
=\frac{(P-\sqrt{P})^2}{P^2-P}=\frac{(\sqrt{P}-1)^2}{P-1}=\frac{(\sqrt{P}-1)^2}{a^2}, \]
and obtains
\[ \tau_1(\beta)=\pm \frac{(\sqrt{P}-1)}{a} \beta,\quad  \text{and} \quad \tau_1(K)\subset K. \]
Similarly, one computes $\tau_2(\beta)^2/\beta^2=(\sqrt{Q}-1)^2/Pb^2$ and obtains
\[ \tau_2(\beta)=\pm \frac{(\sqrt{Q}-1)}{\sqrt{P} a} \beta,\quad  \text{and} \quad \tau_2(K)\subset K. \]
It follows that $K/\Q$ is Galois. 
Let $\tau_1,\tau_2\in \Gal(K/\Q)$ be defined as in \eqref{eq:tau12}.
One easily computes $\tau_1^2(\beta)=-\beta, \tau_2^2(\beta)=-\beta,  
 \tau_1 \tau_2(\beta)=\iota \tau_2\tau_1(\beta)$ and obtains the relations
\[ \tau_1^4=\tau_2^4=1,  \quad \tau_1^2=\tau_2^2,  \quad \tau_1 \tau_2 \tau_1^{-1}=\tau_2^{-1}, \]
showing that $K$ is a $Q_8$-CM field. 

Denote by $D_\ell^+$ the decomposition group at $\ell$ in $G^+=\Gal(K^+/\Q)$. Then $D_\ell=Q_8$ if and only if $D_\ell^+=G^+$. If $\ell\nmid PQ$, then $\ell$ is unramified in $K^+$ and $D_\ell^+$ is cyclic. Thus, $D_\ell^+$ cannot be $G^+$, $D_\ell\neq Q_8$ and $D_\ell$ is cyclic. 
If $\ell|P$, then $ ( \frac{Q}{\ell} ) = ( \frac{1}{\ell} )=1$. So $\ell$ is ramified in $\Q(\sqrt{P})$ and splits in $\Q(\sqrt{Q})$. 
Thus, $D_\ell^+ \neq G^+$ and hence $D_\ell$ is cyclic. It remains to treat the case $\ell|Q$. If $(\frac{P}{\ell})=1$, then $D_\ell^+\simeq \Z/2\Z$ and hence $D_\ell\simeq \Z/4\Z$. Otherwise, $\ell$ is ramified in $\Q(\sqrt{Q})$ and inert in $\Q(\sqrt{P})$. One has $D_\ell^+=G^+$ and $D_\ell=Q_8$. \qed
\end{proof}

\begin{prop}\label{prop:tauQ8}
Let $K$ be a CM field which is Galois over $\Q$ with Galois group $Q_8$. Then \begin{equation}
    \label{eq:tauQ8}
    \tau(T^{K,\Q})=\begin{cases}
    1/2, & \text{if every decomposition group of $K/\Q$ is cyclic;} \\
    2, & \text{otherwise}. 
    \end{cases}
\end{equation}   
\end{prop}
\begin{proof}
Put $G=\Gal(K/\Q)\simeq Q_8$ and $N=\<\iota\>$. We shall show that
\begin{equation}\label{eq:H1LQ8}
    H^1(\Lambda)=H^1(G,\Lambda)=\Z/2\Z
\end{equation}
and 
\begin{equation}\label{eq:Sha2LQ8}
    \Sha^2(\Lambda)=\begin{cases}
    \Z/2\Z \oplus \Z/2\Z, & \text{if every decomposition group of $K/\Q$ is cyclic;} \\
    0, & \text{otherwise}. 
    \end{cases}
\end{equation}
Then \eqref{eq:tauQ8} follows from \eqref{eq:H1LQ8} and \eqref{eq:Sha2LQ8}.

By Proposition~\ref{prop:H1L} and \eqref{eq:4term_es}, we have an exact sequence
\begin{equation}
\begin{tikzcd}
 0 \arrow{r} & H^1(\Lambda) \arrow{r}  & \
 H^1(\Lambda^1)\simeq N^{\rm ab,\vee} \arrow{r}{ \Ver_{G,N}^{\vee}}  & H^2(\Z)=G^{\rm ab,\vee} \arrow{r} & H^2(\Lambda) \arrow{r} & 0,
\end{tikzcd}     
\end{equation}
noting that $H^2(G, \Lambda^1)=0$ when $K/\Q$ is a Galois CM field. Since the $2$-Sylow subgroup of $Q_8$ is not cyclic, the transfer $\Ver_{G,N}$ is not surjective by Proposition~\ref{prop:rued}. Therefore, $\Ver_{G,N}^{\vee}$ is not injective and is zero.
This proves $H^1(\Lambda)\simeq \Z/2\Z$ and $H^2(\Z)'\simeq \Sha^2(\Lambda)$ from \eqref{eq:formula-2}. 

If there is a finite place of $K$ whose decomposition group is not cyclic, then we have $\Sha^2(\Lambda)=0$. Therefore, $\tau(T^{K,\Q})=2$. 
On the other hand, suppose that every decomposition group of $K/\Q$ is cyclic. One has $G^{\rm ab}=G/N\simeq \Z/2\Z \times\Z/2\Z$. Recall
\[ H^2(\Z)'=\{f\in \Hom(G, \Q/\Z): (*)\ f|_D \in {\rm Im} \Ver_{D,D\cap N}^{\vee}, \ \forall\, D\in \scrD \}. \]
Note that $f|_N=0$, so the restriction to $D$ map $f|_D$ is contained in $\Hom(D/N, \Q/\Z)$ whenever $N\subset D$ (also noting that $N\subset D$ if $D$ is not trivial). So if $D=0$ or $D\simeq \Z/2\Z$,
then $f|_D=0$ and condition $(*)$ above is satisfied automatically. Suppose $D\simeq \Z/4\Z$. Since the $2$-Sylow subgroup of $D$ is cyclic, the map $\Ver_{D,N}^\vee:N^\vee \to D^\vee$ is injective and ${\rm Im} \Ver_{D,N}^\vee$ the unique subgroup of $D^\vee$ of order $2$. Then $f|_D\in \Hom(D/N, \Q/\Z)={\rm Im} \Ver_{D,N}^\vee$ and condition $(*)$ is also satisfied automatically. Therefore, $\Sha^2(\Lambda)\simeq H^2(\Z)'=H^2(\Z)\simeq \Z/2\Z \oplus \Z/2\Z$. This proves the proposition. \qed
\end{proof}

\begin{cor} \label{rem:tauQ8}
Let $P$ and $Q$ be two odd positive integers, and $K=\Q(\sqrt{\alpha})$ with $\alpha=-(P+\sqrt{P})(Q+\sqrt{Q})$ be the Galois CM with group $Q_8$ as in Proposition~\ref{prop:Q8CM}. Then
\begin{equation*}
    \tau(T^{K,\Q})=\begin{cases}
    1/2, & \text{if $\left ( \frac{P}{q} \right) =1$ for all primes  $q|Q$;} \\
    2, & \text{otherwise}. 
    \end{cases}
\end{equation*}
\end{cor}

\begin{proof}
By Proposition \ref{prop:Q8CM}, the cyclicity of all decomposition groups of $K/\Q$ is equivalent to the condition $\left(\frac{P}{q}\right)=1$ for all primes $q\mid Q$. Hence the assertion follows from Proposition \ref{prop:tauQ8}. \qed
\end{proof}

\subsection{The dihedral case} Let $D_n$ denote the dihedral group of order $2n$. 

\begin{lemma}\label{lm:density}
  Suppose the CM-field $K$ is Galois over $\Q$ with Galois group $G$, 
  and let  
\[ S:=\{\sigma\in G : \iota \not\in \<\sigma\>\}. \]
  If $|S|\ge |G|/2$, then $n_K\le 2$. If the inequality is strict holds, then $n_K=1$.
\end{lemma}
\begin{proof}
This is due to Jiangwei Xue. Let $L/\Q$ be the class field corresponding to the open subgroup $\Q^\times N(T(\A))\subset \A^\times$ by class field theory. Then $n_K=[\A^\times: \Q^\times N_{L/\Q}(\A_L^\times)]=[L:\Q]$. Suppose $p$ is a rational prime unramified in $K/\Q$ such that the Artin symbol $(p, K/\Q)$ lies in $S$. Since $p$ splits completely in the fixed field $E=K^{D_p}$ of the decomposition group $D_p=\<(p,K/\Q)\>$ of $G$ at $p$ and $\<\iota\>\cap D_p=\{1\}$ (by the definition of $S$), one has $K=K^+ E$ and that every prime $v$ of $K^+$ lying above $p$ splits in $K$, and therefore  $\Qp^\times \subset N(T(\A)) \subset Q^\times N_{L/\Q}(\A_L^\times)$. 
It follows from class field theory that $p$ splits completely in $L$. Thus, the density of $S$ is less than or equal to that of primes splitting completely in $L$. 
By the Chebotarev density theorem, one yields $|S|/|G| \le 1/[L:\Q]$. Therefore, $n_K\le |G|/|S|$ and the assertions then follow immediately. \qed 
\end{proof}

\begin{prop}\label{prop:Dn}
  Let $K$ be a Galois CM field with group $G=D_n$ and $T$ the associated CM torus over $\Q$. Then $n$ is even and $\tau(T)=2$.
\end{prop}
\begin{proof}
Write $D_n=\<t,s: t^n=s^2=1, s t s=t^{-1}\>$. One easily sees that the center $Z(D_n)=\{x\in \<t\>: x^2=1\}$ contains an element of order $2$ if and only if $n=2m$ is even. 
Since $\iota$ is central of order $2$ in $D_n$, $n$ is even.
In this case $|S|=2m+1$ and $n_K=1$ by Lemma~\ref{lm:density}. Therefore, $\tau(T)=2$. \qed
\end{proof}

\begin{rem}\label{rem:density}
  The criterion in Lemma~\ref{lm:density} does not help to compute $\tau(T)$ for the case $G=Q_8$ or $G=\Z/2^{n} \Z$. However, these cases have been treated in Propositions~\ref{prop:tauQ8} and \ref{LiRu}, respectively. 
\end{rem}

\section{Some non-simple CM cases}
\label{sec:NS} 

Keep the notation in Section~\ref{sec:P.3}.
Write $N_i:=\Gal(K_i/K_i^+)=\<\iota_i\>$ with involution $\iota_i$ on $K_i/K_i^+$ and $\iota_i^*\in N_i^{\vee}=\Hom(N_i, \Q/\Z)$ for the unique non-trivial element, that is $\iota_i^*(\iota)=1/2 \mod \Z$. Fix a Galois splitting field $L$ of $T=T^{K,\Q}$ and put $G=\Gal(L/\Q)$. 
For a number field $F$, denote by $\Sigma_F:=\Hom_{\Q}(F,\C)=\Hom_{\Q}(F,\Qbar)$ the set of embeddings of $F$ into $\C$. The Galois group $G_\Q:=\Gal(\Qbar/\Q)$ acts on $\Sigma_F$ by $\sigma\cdot \phi=\sigma\circ \phi$ for $\sigma\in G_\Q$ and $\phi\in \Sigma_F$. 
Let $\Phi_i$ be a CM type of $K_i/K_i^+$, that is, $\Phi_i \cap (\Phi_i \circ \iota_i)=\emptyset$ and $\Phi_i \cup (\Phi_i \circ \iota_i)=\Sigma_{K_i}$, 
and for each $g\in G$, set $\Phi_i(g):=\{\phi\in \Phi_i: g \phi\not\in \Phi_i\}.$

\begin{lemma}\label{lm:H1LCM}
   Let $K$ be a CM algebra and $T=T^{K,\Q}$ the associated CM torus over $\Q$. Then the map
\[ \bigoplus_{i=1}^{r} \Ver^{\vee}_{G,N_i}\colon \bigoplus_{i=1}^{r} N_i^{\vee} \to G^{\rm ab, \vee}=\Hom(G,\Q/\Z) \]   
sends each element $\sum_{i=1}^{r} a_i \iota_i^*$, for $a_i\in \Z$, to the element $f=\sum_{i=1}^{r} a_i f_i$, where $f_i\in \Hom(G,\Q/\Z)$ is the function on $G$ given by $f_i(g)=|\Phi_i(g)|/2 \mod \Z$. 
\end{lemma}
\begin{proof}
We first describe the transfer map $\Ver_{G,N_i}: G\to N_i$.
We may regard the CM fields $K_i$ as subfields of $\Qbar$, and put $\wt X_i:=G/H_i=\Sigma_{K_i}$ and $X_i:=G/\wt N_i=\Sigma_{K_i^+}$. Fix a section $\wt \varphi:X_i \to G$ such that the induced section $\varphi: X_i \to \wt X_i$ has image $\Phi_i$. Since $N_i=\wt N_i/H_i$ is abelian, modulo $D(\wt N_i) H_i$ is the same as modulo $H_i$. Following the definition of $\Ver_{G,N_i}$, for each $g\in G$,
\[ \Ver_{G,N_i}(g)=\prod_{x\in X_i} n^{\wt \varphi}_{g,x} \mod H_i=\iota_i^{n_i(g)}\in N_i, \]
where the element $n^{\wt \varphi}_{g,x}\in \wt N_i$ is defined in Section~\ref{sec:T.4} and $n_i(g):=|\{x\in X_i: n^{\wt \varphi}_{g,x}\not\in H_i\}|$. 
Let $f_i:=\Ver_{G,N_i}^\vee(\iota_i^*)$. 
Then (cf.~Lemma~\ref{lm:ver-dual})
\[ f_i(g)=\iota_i^*(\Ver_{G,N_i} (g))=\iota_i^*(\iota_i^{n_i(g)})=n_i(g)/2 \in \Q/\Z. \]
Thus, it remains to show $n_i(g)=|\Phi_i(g)|$. Let
$\wt \Phi_i:=\wt \varphi(\Sigma_{K_i^+})$. Then we have bijections $\wt \Phi_i\isoto \Phi_i \isoto \Sigma_{K_i^+}$. Let $x\in \Sigma_{K_i^+}$, and let $\wt \phi\in\wt \Phi_i$ and $\phi\in \Phi_i$ be the corresponding elements. Put 
$\wt \phi':=\wt \varphi (gx)$ and $\phi'=\varphi(gx)$, the image of $\wt \phi'$. Then $g\wt \phi$ and $\wt \phi'$ are elements lying over $gx$ and we have $g \wt \phi=\wt \phi' {n}^{\wt \varphi}_{g,x}$. So $g\phi=\phi'$ if and only if ${n}^{\wt \varphi}_{g,x}\in H_i$. On the other hand, since $\phi'\in \Phi_i$ and two elements $g \phi$ and $\phi'$ are lying over the same element $gx$, we have
\[ g\phi\not \in \Phi_i \iff g \phi\neq \phi' \iff {n}^{\wt \varphi}_{g,x}\not\in H_i. \]
Therefore, the bijection $\Phi_i\isoto \Sigma_{K_i^+}$ gives the bijection of subsets:
\[ \{\phi\in \Phi_i: g\phi \not \in \Phi_i\}\isoto \{x\in \Sigma_{K_i^+}: {n}^{\wt \varphi}_{g,x}\not\in H_i\}. \]
This shows the desired equality $|\Phi_i(g)|=n_i(g)$.
\qed
\end{proof}

By Proposition~\ref{prop:H1L} and Lemma~\ref{lm:H1LCM}, we give an independent proof of the following result of Li and R\"ud \cite[Proposition A.11]{achter-altug-gordon}.

\begin{cor}\label{cor:A.11}
  Let $K$ and $T$ be as in Lemma~\ref{lm:H1LCM} and $\Lambda=X(T)$ be the character group of $T$. Then 
\begin{equation}\label{eq:A.11}
H^1(\Lambda)\simeq \left \{(a_i)\in \{\pm 1\}^r: \sum_{1\le i \le r, a_i=-1} |\Phi_i(g)| \in 2\Z, \ \forall\, g\in G \right \}.    
\end{equation}
\end{cor}

\begin{proof}
Indeed after making the identity $\oplus_{i=1}^{r} N_i^\vee\simeq (\Z/2\Z)^r$ with $\{\pm 1\}^r$, by Lemma~\ref{lm:H1LCM} the map $\oplus_{i=1}^{r} \Ver^\vee_{G,N_i}\colon \oplus_{i=1}^{r} N_i^\vee=\{\pm 1\}^r\to \Hom(G,\Q/\Z)$ sends $(a_i)$ to the function $f$ with $f(g)=\sum_{1\le i\le r, a_i=-1} |\Phi_i(g)|/2 \mod \Z$. Thus, $f=0$ precisely when 
$\sum_{1\le i \le r, a_i=-1} |\Phi_i(g)| \in 2\Z, \ \forall\, g\in G$, and \eqref{eq:A.11} follows from Proposition~\ref{prop:H1L}. \qed
\end{proof}


\begin{prop}
\label{prop:nKH2Z}
  Let $K$ be a CM algebra and $T$ the associated CM torus over $\Q$. Then there is an isomorphism $\A^\times/ \Q^\times N(T(\A))  \simeq H^2(\Z)'^\vee$.
\end{prop}
\begin{proof}
By \cite[Theorem 1.2.9]{rosengarten}, we have the following commutative diagram with row exact sequence
\def\pro{\rm pro}
\begin{equation}
\begin{CD}
0 @>>> T(\Q)_{\pro} @>>> T(\A)_{\pro} @>>> H^2(\Q, \wh T)^\vee @>>> \Sha^1(T) @>>> 0 \\ 
@. @VV{N}V @VV{N}V @VV{{\wh N}^\vee}V @VVV \\
0 @>>> (\Q^\times)_{\pro} @>>> (\A^\times)_{\pro} @>>> H^2(\Q, \Z)^\vee @>>> 0, \\ 
\end{CD}    
\end{equation}
where $A_{\pro}$ denotes the profinite completion of an abelian group $A$. It follows from the Poitou-Tate duality and class field theory that
\begin{equation}
\begin{tikzcd}
\qquad (T(\A)/T(\Q))_{\pro} \arrow[twoheadrightarrow]{r} \arrow{d}{N} & (H^2(\Q, \wh T)/\Sha^2(\Q,\wh T))^\vee \arrow{d}{{\wh N}^\vee} \\ 
(\A^\times/\Q^\times)_{\pro} \arrow{r}{\sim} & H^2(\Q,\Z)^\vee 
\end{tikzcd}
\end{equation} 
By definition $H^2(\Q,\Z)'=\Ker \left(\wh N: H^2(\Q,\Z) \to H^2(\Q, \wh T)/\Sha^2(\Q,\wh T) \right)$, so we have
\[ \Coker N=(T(\A)/T(\Q) N(T(\A)))_{\pro} \isoto \Coker \wh N^\vee=H^2(\Q,\Z)'^\vee. \]
Since $T(\A)/T(\Q) N(T(\A))$ is finite, it is equal to its profinite completion. This proves the proposition. \qed 
\end{proof}
\begin{remark}
Proposition~\ref{prop:nKH2Z} gives a cohomological interpretation of the group $\A^\times/ \Q^\times N(T(\A))$. This reminds the main theorem $\A_k^\times/N_{K/k}(\A_K^\times) k^\times \simeq H^2(G,\Z)^\vee$ of class field theory, where $K/k$ is Galois with Galois group $G$. 
\end{remark}

\begin{lemma}\label{lm:prod}
  Let $K_1$ be a CM field and $K:=K_1^r$ be the $r$-copies of $K_1$. Then
  $\tau(T^{K,\Q})=2^{r-1}\cdot \tau(T^{K_1,\Q})$. 
\end{lemma}

\begin{proof}
  Observe that $N(T^{K,\Q}(\A))=N(\A_K^\times) \cap \A^\times$, where $N=N_{K/K^+}$. To see this, the inclusion $\subseteq$ is clear. If $x\in N(\A_K^\times) \cap \A^\times$, then $x=N(y)$ for some $y\in \A_K^\times$. By definition $y\in T(\A)$, and hence $x\in N(T(\A))$. This verifies the other inclusion. It follows that if $K=\prod_{i=1}^r K_i$ is a product of CM fields $K_i$, then
\[ N(T^{K,\Q}(\A))=\A^\times \bigcap^r_{i=1} N_{K_i/K_i^+}(\A_{K_i}^\times) \]  
in the sense that the right hand side consists of all elements $x\in \A^\times$ which are contained in $N_{K_i/K_i^+}(\A_{K_i}^\times)$ via the embeddings $\A^\times \embed \A_{K_i^+}^\times$ for all $i$. Thus, if $K_i=K_1$ for all $i$, then
$N(T^{K,\Q}(\A))=\A^\times \cap  N_{K_1/K_1^+}(\A_{K_1}^\times)=N(T^{K_1,\Q}(\A))$ and hence $n_K=n_{K_1}$. By Proposition~\ref{formula}, $\tau(T^{K,\Q})=2^{r}/n_K=2^{r-1}\cdot 2/n_{K_1}= 2^{r-1}\cdot \tau(T^{K_1})$. \qed  
\end{proof}

\begin{cor}\label{cor:range}
  For any integer $n\ge 0$, there exists a CM torus $T$ over $\Q$ such that
  $\tau(T)=2^n$.
\end{cor}
\begin{proof}
  Take $K=E^r$ with $r=n+1\ge 1$, where $E$ is an imaginary quadratic field. Then $\tau(T^{K,\Q})=2^{r-1} \cdot \tau(T^{E,\Q})=2^{n}$. \qed
\end{proof}

\begin{prop}\label{prop:prod-Q8}
  Suppose that CM fields $K_{1},\ldots,K_{r}$ satisfy the following:
  \begin{enumerate}
  \item[\emph{(a)}]$K_i$ is Galois over $\Q$ with group $G_i\simeq Q_8$ for any $i$;
  \item[\emph{(b)}] the Galois group $G=\Gal(L/\Q)$ of the compositum $L=K_1\cdots K_r$ over $\Q$ is isomorphic to $G_1\times \dots \times G_r$;
  \item[\emph{(c)}] every decomposition group of $G_i$ is cyclic for all $1\le i\le r$. 
  \end{enumerate}
  Then $\tau(T^{K,\Q})=(1/2)^r$. 
\end{prop}

\begin{proof}
By the condition (c) and Proposition \ref{prop:tauQ8}, one has
\begin{equation}\label{eq:tauT1/2}
\tau(T^{K_i,\Q})=1/2
\end{equation}
for all $1\leq i\leq r$. Moreover, the equality $[K_i:K_i^{+}]=2$ and the conditions (a), (b) and (c) imply that the CM fields $K_1,\ldots,K_{r}$ satisfies the assumptions in Proposition \ref{prop:product-cyclic-dec}. See also the conditions (i) and (ii) in Section \ref{sec:MC} and the beginning of Subsection \ref{subsec:tauT}. Hence the assertion follows from Proposition \ref{prop:product-cyclic-dec} and \eqref{eq:tauT1/2}. \qed
\end{proof}

\begin{thm}\label{thm:family-Q8}
For any positive integer $r$, there exist $Q_8$-CM fields $K_i$ for $1\le i \le r$ that satisfy the conditions (a), (b) and (c) in Proposition~\ref{prop:prod-Q8}.  
\end{thm}

We will give the proof in the next section. By Proposition~\ref{prop:prod-Q8} and Theorem~\ref{thm:family-Q8}, we prove the following result.

\begin{cor}\label{cor:main}
  For any integer $n$, there exists a CM torus $T$ over $\Q$ such that $\tau(T)=2^n$.
\end{cor}

\begin{rem}\label{rem:ker1}
  In \cite{kottwitz:jams1992} Kottwitz computed the Hasse-Weil zeta function of
 the moduli spaces $S_{K^p}$ of PEL-type. 
 Using the notation there, it is shown in Section 8 that the algebraic variety $S_{K^p}$ over the reflex field $E$ is a finite disjoint union of the canonical model of Shimura varieties associated to the Shimura datum $(G, h^{-1}, K^p)$ indexed by 
 $\ker^1(\Q,G):=\ker (H^1(\Q,G)\to \prod_{v\le \infty} H^1(\Q_v,G))$. In Case C and Case A with $n$ even, the set $\ker^1(\Q,G)$ is trivial and there is no difference between the moduli space $S_{K^p}$ and the canonical model of the Shimura variety in question. In Case A with $n$ odd, the set $\ker^1(\Q,G)$ is canonically isomorphic to $\ker^1(\Q,Z)$, where $Z$ is the kernel of the map $F^\times \times \Q^\times \to F_0^\times$ sending $(x,t)$ to $N_{F/F_0}(x)t^{-1}$ and $F$ is the center of the central simple $\Q$-algebra $B$ in the input PEL-datum. The $\Q$-torus $Z$ is exactly the CM torus associated to the CM field $F$ and $\ker^1(\Q,Z)$ is its Tate-Shafarevich group.  
\end{rem}

\begin{que} Is Proposition~\ref{LiRu} (2) still true if one drops the condition that $K/\Q$ is Galois?
\end{que}

\section{Construction of an effective family of $Q_8$-CM fields} \label{sec:Q8}

\subsection{Existence of $Q_8$-extensions of fields}

Let $k$ be a field of characteristic different from $2$, and denote by $\Br(k)$ the Brauer group of $k$. Then we define a pairing
\begin{equation*}
(\,,\,)_{k}\colon k^{\times}\times k^{\times}\rightarrow \Br(k)
\end{equation*}
by sending $(a,b)\in k^{\times}\times k^{\times}$ to the Brauer class of the quaternion algebra
\begin{equation*}
\left(\frac{a,b}{k}\right):=k\oplus k\alpha \oplus k\beta \oplus k\alpha\beta,
\end{equation*}
where $\alpha^{2}=a$, $\beta^{2}=b$ and $\beta \alpha=-\alpha\beta$. By definition, the pairing $(\,,\,)_{k}$ is symmetric, and the image of $(\,,\,)_{k}$ is contained in the $2$-torsion group of $\Br(k)$. 

For $a\in k^{\times}$, put
\begin{equation*}
k_{a}:=k[T]/(T^{2}-a). 
\end{equation*}
We denote by $N_{k_{a}/k}\colon k_{a} \rightarrow k$ the norm map of $k_{a}/k$. 

\begin{prop}[{\cite[Proposition 1.1.7]{MR3727161}}]\label{prop:quat}
Let $a,b\in k^{\times}$. Then the following are equivalent: 
\begin{enumerate}
\item[\emph{(1)}] $(a,b)_{k}=1$;
\item[\emph{(2)}] $a\in N_{k_{b}/k}(k_{b}^{\times})$;
\item[\emph{(3)}] $b\in N_{k_{a}/k}(k_{a}^{\times})$. 
\end{enumerate}
\end{prop}

The following is one of the most important key to prove Theorem \ref{thm:family-Q8}. 

\begin{thm}[{\cite[Theorem 4]{MR1080998}}]\label{thm:Kiming}
Let $E=k(\sqrt{a},\sqrt{b})$ be a biquadratic extension of $k$, where $a,b\in k^{\times}$. Then the following are equivalent:
\begin{enumerate}
\item[\emph{(1)}] there is a quadratic extension of $K/E$ such that $K/k$ is a $Q_{8}$-extension;
\item[\emph{(2)}] $(a,a)_{k}(b,b)_{k}(a,b)_{k}=1$. 
\end{enumerate}
\end{thm}

\subsection{Proof of Theorem \ref{thm:family-Q8}}

First, we recall the properties on the pairings $(\,,\,)_{0}:=(\,,\,)_{\Q}$ and $(\,,\,)_{v}:=(\,,\,)_{\Q_{v}}$ for all places $v$ of $\Q$. 

\begin{prop}\label{prop:norm}
Let $v$ be a place of $\Q$, $a,b\in \Q_{v}^{\times}$ and let $\Q_{v,a}:=\Q_v[T]/(T^2-a)$. 
\begin{enumerate}
\item[\emph{(1)}] If $v$ is infinite, then we have $(a,b)_{v}=1$ if and only if either $a$ or $b$ is positive. 
\item[\emph{(2)}] Assume that $v=\ell$ is a non-dyadic finite place. If $a,b\in \Z_{\ell}^{\times}$, then $(a,b)_{\ell}=1$. 
\item[\emph{(3)}] Under the assumption on $v$ in {\rm (2)}, if $a=\ell$ and $b\in \Z_{\ell}^{\times}$, then we have $(a,b)_{\ell}=\left(\frac{b}{\ell}\right)$. 
\item[\emph{(4)}] If $v=2$ and $a,b\in 1+4\Z_2$, then one has $(a,b)_{2}=1$. 
\end{enumerate}
\end{prop}

\begin{proof}
(1) This follows from Proposition \ref{prop:quat} and the equality $N_{\C/\R}(\C^{\times})=\R_{>0}$. 

(2) Since $\ell\neq 2$ and $a\in \Z_{\ell}^{\times}$, we have
\begin{equation*}
\Q_{\ell,a}\cong
\begin{cases}
\Ql \times \Ql,&\text{if }a\in (\Z_{\ell}^{\times})^{2};\\
\Q_{\ell^2},&\text{if }a\notin (\Z_{\ell}^{\times})^{2}. 
\end{cases}
\end{equation*}
Hence $N_{\Q_{\ell,a}/\Ql}(\Q_{\ell,a}^{\times})$ contains $\Z_{\ell}^{\times}$. Hence the assertion follows from Proposition \ref{prop:quat} and the assumption $b\in \Z_{\ell}^{\times}$. 

(3) Since $\ell$ is not equal to $2$, we have
\begin{equation*}
N_{\Ql(\sqrt{\ell})/\Q}(\Ql(\sqrt{\ell})^{\times})=\langle -\ell \rangle \times (\Z_{\ell}^{\times})^{2}. 
\end{equation*}
Therefore the assertion follows from Proposition \ref{prop:quat}. 

(4) Since $1+8\Z_2=(\Z_2^{\times})^{2}$, there is an isomorphism 
\begin{equation*}
\Q_{2,a}\cong
\begin{cases}
\Q_2 \times \Q_2, &\text{if }a\in (\Z_{2}^{\times})^{2};\\
\Q_{4}, &\text{if }a\notin (\Z_{2}^{\times})^{2}. 
\end{cases}
\end{equation*}
In particular, $N_{\Q_{2,a}/\Ql}(\Q_{2,a}^{\times})$ contains $\Z_{2}^{\times}$. Hence the assertion follows from $b\in 1+4\Z_2$ and Proposition \ref{prop:quat}. \qed
\end{proof}



The following two lemmas will be used later. 

\begin{lem}\label{lem:cm}
Let $E$ be a totally real field which is Galois over $\Q$. 
\begin{enumerate}
\item[\emph{(1)}] Let $K/E$ be a quadratic extension such that $K/\Q$ is Galois. Then $K$ is either totally real or CM. 
\item[\emph{(2)}] If there is a quadratic extension $K/E$ such that $K/\Q$ is Galois, then there is a totally imaginary quadratic extension $K'/E$ such that $K'/\Q$ is Galois and 
\begin{equation*}
\Gal(K'/\Q)\cong \Gal(K/\Q). 
\end{equation*}
\end{enumerate}
\end{lem}

\begin{proof}
(1) If $K$ is not totally real, then it is totally complex since $K/\Q$ is Galois. Fix an embedding $\varepsilon \colon K \hookrightarrow \C$, and let $\iota$ be the element of $\Gal(K/\Q)$ induced by the complex conjugation and $\varepsilon$. Then $\iota$ is the unique non-trivial element of $\Gal(K/E)\subset \Gal(K/\Q)$ since $E$ is totally real. On the other hand, the assumption that $E$ is Galois implies that $\Gal(K/E)$ is central in $\Gal(K/\Q)$. Hence $\iota$ is contained in the center of $\Gal(K/\Q)$, which implies that $K$ is a CM field. 

(2) By (1), we may assume that $K$ is totally real. Write $K=E(\sqrt{\alpha})$, where $\alpha \in E^{\times}$, and put $K':=E(\sqrt{-\alpha})$. Note that $K'$ corresponds to $\langle(\varepsilon,\iota)\rangle$ under the isomorphism
\begin{equation*}
\Gal(K(\sqrt{-1})/\Q)\cong \Gal(K/\Q)\times \Gal(\Q(\sqrt{-1})/\Q). 
\end{equation*}
Here $\varepsilon$ is the unique non-trivial element of $\Gal(K/E)$, and $\iota$ is the complex conjugation on $\Q(\sqrt{-1})$. Note that $\varepsilon$ is central in $\Gal(K/\Q)$. Then $K'$ is not totally real $K'$ is Galois over $\Q$, and hence it is CM by (1). Moreover, the composite
\begin{equation*}
\Gal(K/\Q)\xrightarrow{g\mapsto (g,\id_{\Q(\sqrt{-1})})}\Gal(K/\Q)\times \Gal(\Q(\sqrt{-1})/\Q) \cong \Gal(K(\sqrt{-1})/\Q)\rightarrow \Gal(K'/\Q)
\end{equation*}
is an isomorphism by the definition of $K'$. \qed
\end{proof}

For a number field $E$ which is finite Galois over $\Q$, we write for $\Ram(E)$ the set of prime numbers which ramify in $E$. 

\begin{lem}\label{lem:disjoint}
Let $r$ be a positive integer, and $K_{1},\ldots,K_{r}$ be number fields which are Galois over $\Q$. For each $i$, we denote by $E_{i}$ the maximal abelian subfield of $K_{i}$. Assume the following: 
\begin{enumerate}
\item[\emph{(i)}] $[K_{i}:E_{i}]$ is a prime number for any $i \ge 1$;
\item[\emph{(ii)}] $\mathrm{Ram}(E_{i})\cap \mathrm{Ram}(E_{j})=\emptyset$ if $i\neq j$. 
\end{enumerate}
Let $L$ be the compositum of $K_1,\ldots,K_r$. Then there is an isomorphism
\begin{equation*}
\Gal(L/\Q)\cong \prod_{i=1}^{r}\Gal(K_{i}/\Q). 
\end{equation*}
\end{lem}

\begin{proof}
We give a proof by induction on $r$. It is trivial if $r=1$. Next, assume that the assertion holds for $r-1$, that is, there is an isomorphism
\begin{equation}\label{eq:indh}
\Gal(L'/\Q)\cong \prod_{i=1}^{r-1}\Gal(K_i/\Q),
\end{equation}
where $L':=K_{1}\cdots K_{r-1}$. We first prove the equality
\begin{equation}\label{eq:Er}
    E_{r}\cap L'=\Q. 
\end{equation}
Since $E_{r}$ is abelian over $\Q$, it is contained in the maximal abelian subfield $E'$ of $L'$. On the other hand, the induction hypothesis \eqref{eq:indh} implies the equality $E'=E_1\cdots E_{r-1}$. Hence $E_{r}\cap L'$ is contained in $E_{r}\cap (E_1\cdots E_{r-1})$. However, we have $E_{r}\cap (E_1\cdots E_{r-1})=\Q$ by the assumption (ii) and the global class field theory, and hence \eqref{eq:Er} holds. 

Now we prove the equality $K_{r}\cap L'=\Q$, which gives the desired assertion. If $K_{r}\cap L'\neq \Q$, then we have $E_{r}\cap (K_{r}\cap L')=\Q$ and $E_{r}\subsetneq E_{r}\cdot(K_{r}\cap L')\subset K_{r}$ by \eqref{eq:Er}. Therefore $K_{r}$ is equal to the compositum of $E_{r}$ and $K_{r}\cap L'$ by the assumption (i). Consequently, there is an isomorphism
\begin{equation*}
\Gal(K_{r}/\Q)\cong \Gal(E_{r}/\Q)\times \Gal(K_{r}\cap L'/\Q). 
\end{equation*}
In particular, $[K_{r}\cap L':\Q]=[K_{r}:E_{r}]$ is a prime number, and hence $K_{r}/\Q$ is abelian. This contradicts the assumption (i), which implies the desired equality $K_{r}\cap L'=\Q$. \qed
\end{proof}

Let $\mathbf{L}$ be the set of unordered pairs of prime numbers $\{\ell,\ell'\}$ satisfying the following: 
\begin{equation*}
\ell \equiv \ell' \equiv 1\bmod 4,\quad \left(\frac{\ell'}{\ell}\right)=1. 
\end{equation*}

\begin{prop}\label{prop:Dirichlet}
For any positive integer $r$, there are $r$-pairs $\{\ell_1,\ell'_1\},\ldots,\{\ell_{r},\ell'_{r}\}$ in $\mathbf{L}$ such that 
\begin{equation*}
|\{\ell_1,\ldots,\ell_{r},\ell'_1,\ldots,\ell'_{r}\}|=2r. 
\end{equation*}
\end{prop}

\begin{proof}
By the Dirichlet prime number theorem, there are $r$ distinct prime numbers $\ell_1,\ldots,\ell_{r}$ which are congruent to $1$ modulo $4$. Moreover, the Dirichlet prime number theorem implies the existence of prime numbers $\ell'_1,\ldots,\ell'_{r}$ satisfying the following: 
\begin{equation*}
\begin{cases}
\ell'_{i}\equiv 1\bmod 4,\,\left(\frac{\ell'_{i}}{\ell_{i}}\right)=1&\text{if }i=1,\\
\ell'_{i}\equiv 1\bmod 4,\,\left(\frac{\ell'_{i}}{\ell_{i}}\right)=1,\,\ell'_{i}\notin \{\ell'_1,\ldots,\ell'_{i-1}\}&\text{if }i \geq 2. 
\end{cases}
\end{equation*}
Therefore the assertion holds. \qed
\end{proof}

For $\lambda:=\{\ell,\ell'\} \in \mathbf{L}$, put $K_{\lambda}^{+}:=\Q(\sqrt{\ell},\sqrt{\ell'})$. 

\begin{lem}\label{lem:biquad}
\emph{For any $\lambda \in \mathbf{L}$, the decomposition groups of $K_{\lambda}^{+}/\Q$ at all finite places are cyclic. }
\end{lem}

\begin{proof}
Write $\lambda=\{\ell,\ell'\}$. Let $v$ be a finite place of $K_{\lambda}^{+}$ which lies above a prime number $\ell_{0}$. If $\ell_{0}\notin \{\ell,\ell'\}$, then $K_{\lambda}^{+}/\Q$ is unramified at $v$, and hence the assertion holds for $v$. The assertion for $\ell_{0}=\ell$ follows from the assumption $\left(\frac{\ell'}{\ell}\right)=1$. Finally, in the case $\ell_{0}=\ell'$, the statement is a consequence of the equality $\left(\frac{\ell}{\ell'}\right)=\left(\frac{\ell'}{\ell}\right)=1$. \qed
\end{proof}

\begin{prop}\label{prop:exist-Q8}
For any $\lambda \in \mathbf{L}$, there is a CM field $K$ containing $K_{\lambda}^{+}$ such that $K/\Q$ is a $Q_{8}$-extension. 
\end{prop}

\begin{proof}
By Lemma \ref{lem:cm} (2), it suffices to prove the existence of $Q_8$-extension $K$ of $\Q$ containing $K_{\lambda}^{+}$. Write $\lambda=\{\ell,\ell'\}$. It is equivalent to the equality
\begin{equation*}
(\ell,\ell)_{0}(\ell',\ell')_{0}(\ell,\ell')_{0}=1, 
\end{equation*}
which is a consequence of Theorem \ref{thm:Kiming}. Since the image of the class $(\ell,\ell)_{0}(\ell',\ell')_{0}(\ell,\ell')_{0}$ in $\Br(\Q_v)$ is equal to $(\ell,\ell)_{v}(\ell',\ell')_{v}(\ell,\ell')_{v}$, by 
the Albert--Brauer--Hasse--Noether theorem~\cite[Chapter IX, Section 6, p.~195]{Lang-ANT} 
it is equivalent to prove the following for any place $v$ of $\Q$:
\begin{equation}\label{eq:pairing}
(\ell,\ell)_{v}(\ell',\ell')_{v}(\ell,\ell')_{v}=1.
\end{equation}

\textbf{Case 1.~$v$ is infinite. }
In this case, \eqref{eq:pairing} follows from Proposition \ref{prop:norm} (1) since $\ell$ and $\ell'$ are positive. 

\textbf{Case 2.~$v=\ell_{0}\notin \{2,\ell,\ell'\}$. }
The condition $\ell_{0}\notin \{\ell,\ell'\}$ derives that $\ell$ and $\ell'$ are units in $\Z_{\ell_{0}}$. Hence, since $\ell_{0} \neq 2$, one has
\begin{equation*}
(\ell,\ell)_{\ell_{0}}=(\ell',\ell')_{\ell_{0}}=(\ell,\ell')_{\ell_{0}}=1 
\end{equation*}
by Proposition \ref{prop:norm} (2). This implies equality (\ref{eq:pairing}). 

\textbf{Case 3.~$v=2$. }
Since $\ell \equiv \ell'\equiv 1\bmod 4$, Proposition \ref{prop:norm} (4) implies
\begin{equation*}
(\ell,\ell)_{2}=(\ell',\ell')_{2}=(\ell,\ell')_{2}=1.  
\end{equation*}
Hence \eqref{eq:pairing} holds. 

\textbf{Case 4.~$v=\ell$. }
The assumption $\ell \equiv 1\bmod 4$ is equivalent to the equality $\left(\frac{-1}{\ell}\right)=1$. Hence
\begin{equation*}
(\ell,\ell)_{\ell}=(\ell,-1)_{\ell}=1. 
\end{equation*}
On the other hand, the assumption $\left(\frac{\ell'}{\ell}\right)=1$ implies
\begin{equation*}
(\ell,\ell')_{\ell}=(\ell',\ell')_{\ell}=1,
\end{equation*}
which is a consequence of Proposition \ref{prop:norm} (2) and (3). Therefore we obtain the desired equality (\ref{eq:pairing}). 

\textbf{Case 5.~$v=\ell'$. }
Since $\ell'\equiv 1\bmod 4$ and
\begin{equation*}
\left(\frac{\ell}{\ell'}\right)=\left(\frac{\ell'}{\ell}\right)=1,
\end{equation*}
the assertion \eqref{eq:pairing} follows from the same argument as Case 4. \qed
\end{proof}

In the following, we give a proof of Theorem \ref{thm:family-Q8}. 
By Corollary \ref{cor:range}, we may assume $n=-r<0$. Take $\lambda_1=\{\ell_1,\ell'_1\},\ldots,\lambda_{r}=\{\ell_{r},\ell'_{r}\} \in \mathbf{L}$ satisfying
\begin{equation}\label{eq:primes}
|\{\ell_1,\ldots,\ell_{r},\ell'_1,\ldots,\ell'_{r}\}|=2r,
\end{equation}
which is possible by Proposition \ref{prop:Dirichlet}. Then, Proposition \ref{prop:exist-Q8} implies that there is a $Q_{8}$-CM field containing $K_{\lambda_{i}}^{+}$ for any $1\le i \le r$. 

The following immediately implies the desired assertion. 

\begin{thm}
Under the above notations, let $K_{\lambda_{i}}$ be a $Q_{8}$-CM field containing $K_{\lambda_{i}}^{+}$ for each $1\le i\le r$. Then the CM fields $K_{\lambda_1},\ldots,K_{\lambda_{r}}$ satisfy the conditions (a), (b) and (c) in Proposition \ref{prop:prod-Q8}. 
\end{thm}

\begin{proof}
From the definition of $K_{\lambda_{i}}$ for $1\leq i\leq r$, condition (a) holds.

We shall show that the assumptions (i) and (ii) in Lemma \ref{lem:disjoint} hold. By Lemma \ref{lm:Q8} (3), for any $1\le i\le r$, $K_{\lambda_{i}}^{+}$ is the maximal abelian subfield of $K_{\lambda_{i}}$ and $[K_{\lambda_{i}}:K_{\lambda_{i}}^{+}]=2$. In particular, assumption (i) holds. On the other hand, since $\Ram(K_{\lambda_{i}}^{+})=\lambda_{i}$ for any $1\leq i\leq r$, the condition \eqref{eq:primes} implies the equality $\Ram(K_{\lambda_{i}}^{+})\cap \Ram(K_{\lambda_{j}}^{+})=\emptyset$ for $i\neq j$. Hence we obtain the assumption (ii), which verifies 
condition (b). 

 Take a finite places $w$ of $K_{\lambda_{i}}$ and $v$ of $K_{\lambda_{i}}^{+}$ satisfying $w\mid v$. Then Lemma \ref{lem:biquad} implies the cyclicity of the decomposition group of $\Gal(K_{\lambda_{i}}^{+}/\Q)$ at $v$. Since $K_{\lambda_{i}}/\Q$ is a $Q_8$-extension, the decomposition group at $w$ is cyclic by Lemma \ref{lm:Q8} (2). Therefore condition (c) holds. \qed
\end{proof}

\section{Products of two linearly disjoint Galois CM fields} \label{sec:P2}

In this section we show the following result.

\begin{thm}\label{thm:prod-ld}
There are infinitely many CM algebras $K=K_1 \times K_2$ with linearly disjoint Galois CM fields $K_1$ and $K_2$ such that 
\begin{equation}\label{eq:prod-ld}
   \tau(T^{K,\Q})=\frac{1}{2} \prod_{i=1}^2 \tau(T^{K_i,\Q}). 
\end{equation}
\end{thm}

This theorem shows that the conclusion of Proposition~\ref{prop:product-cyclic-dec} is no longer true if one drops the cyclicity of decomposition groups of $G_i$ for all $i$. 
We shall use the notations in Section~\ref{subsec:tauT}. In particular, $L=K_1K_2$, $G:=\Gal(L/\Q)$, $G_i:=\Gal(K_i/\Q)$, 
$H_i:=\Gal(L/K_i)$, $\wt N_i:=\Gal(L/K_i^+)$ and $N_i=\Gal(K_i/K_i^+)$ for $i=1,2$.

First, we give a sufficient condition on $K=K_1\times K_2$ for which Theorem \ref{thm:prod-ld} holds. Let $\scrC$ and $\scrC_{i}$ be the sets of cyclic subgroups of $G$ and $G_i$ respectively. Then put
\begin{align*}
H^{2}(\Z)''&:=\{f\in G^{\vee}\mid f\!\mid_{D}\in {\rm Im}(\Ver_{D,\overline{\mathbf{D}}}^{\vee})\text{ for all }D\in \scrC \},\\
H^{2}(G_{i},\Z)''&:=\{f\in G_{i}^{\vee}\mid f\!\mid_{D'}\in {\rm Im}(\Ver_{D',D'\cap N_{i}}^{\vee})\text{ for all }D'\in \scrC_{i}\}. 
\end{align*}

\begin{lem}\label{lm:double-p}
If $G\cong G_1\times G_2$, then $H^{2}(\Z)''=H^{2}(G_1,\Z)''\times H^{2}(G_2,\Z)''$. 
\end{lem}

\begin{proof}
The proof is the same as Lemma \ref{lm:(a)(b)}. \qed
\end{proof}

We define a subgroup $D_{0}$ of $(\Z/4\Z \times \Z/2\Z)\times \Z/2\Z$ as follows: 
\begin{equation*}
D_{0}:=\< (\bar 1, \bar 0,\bar 1),(\bar 0 ,\bar 1,\bar 0)\>. 
\end{equation*}
Here we denote by $(\bar a, \bar b, \bar c)$ the element $(a\bmod 4, b\bmod 2, c\bmod 2)$ in $\Z/4\Z\times \Z/2\times \Z/2\Z$ for $a,b,c\in\Z$.
\begin{prop}\label{prop:neg1}
Assume that Galois CM fields $K_1$ and $K_2$ satisfy the following: 
\begin{enumerate}
\item[\emph{(i)}] $G_{1}\cong \Z/4\Z \times \Z/2\Z$, $N_{1}\cong \langle (\bar 2,\bar 0)\rangle\subset G_1$ and $G_{2}\cong \Z/2\Z$;
\item[\emph{(ii)}] $G\cong G_1\times G_2$, that is, $K_1$ and $K_2$ are linearly disjoint;
\item[\emph{(iii)}] $\scrD=\scrC \cup \{D_{0}\}$. 
\end{enumerate}
Put $K:=K_1\times K_2$. Then we have the following: 
\begin{equation*}
\tau(T^{K_1,\Q})=2, \quad \tau(T^{K_2,\Q})=1,\quad \tau(T^{K,\Q})=1. 
\end{equation*}
In particular, the equality \eqref{eq:prod-ld} holds.
\end{prop}

\begin{proof}
By definition, we obtain the following: 
\begin{gather*}
\widetilde{N}_{1}=(2\Z/4\Z \times \{\bar 0 \})\times \Z/2\Z,\quad H_{1}=\{(\bar 0, \bar 0 )\}\times \Z/2\Z,\\
\widetilde{N}_{2}=G,\quad H_{2}=(\Z/4\Z \times \Z/2\Z)\times \{\bar 0 \}. 
\end{gather*}
Moreover, Proposition \ref{prop:Sha2L-1} implies
\begin{equation*}
\tau(T^{K_i,\Q})=\frac{2}{|H^{2}(G_i,\Z)'|},\quad \tau(T^{K_1\times K_2,\Q})=\frac{4}{|H^{2}(\Z)'|}. 
\end{equation*}
Since $G_1$ is $D_0$ modulo $G_2$, by (iii) $G_1$ itself is a decomposition group of $G_1$. By this and that $G_1$ is not cyclic, one computes that $H^{2}(G_1,\Z)'=0$, which implies $\tau(T^{K_1,\Q})=2/1=2$. Moreover, the equality $\tau(T^{K_2,\Q})=1$ follows from Proposition \ref{LiRu} (2) as $[K_2:\Q]=2$. In particular, we obtain $|H^{2}(G_2,\Z)'|=2$, that is,
\begin{equation}\label{eq:H2G2p}
H^{2}(G_2,\Z)'=H^{2}(G_2,\Z)''=G_2^{\vee}. 
\end{equation}
On the other hand, for the equality $\tau(T^{K,\Q})=1$, it suffices to prove
\begin{equation}\label{eq:want1}
H^{2}(\Z)'=\{f\in G^{\vee}\mid f((2\Z/4\Z \times \Z/2\Z)\times \{\bar 0 \})=0\}. 
\end{equation}
By direct computation, we have
\begin{equation*}
H^{2}(G_1,\Z)''=\{f_1\in G_1^{\vee}\mid f(2\Z/4\Z \times \Z/2\Z)=0\}. 
\end{equation*}
Combining this equality, \eqref{eq:H2G2p} and Lemma \ref{lm:double-p}, one has
\begin{equation*}
H^{2}(\Z)'\subset H^{2}(\Z)''=H^{2}(G_1,\Z)''\times H^{2}(G_2,\Z)''=\{f\in G^{\vee}\mid f((2\Z/4\Z \times \Z/2\Z)\times \{\bar 0 \})=0\}. 
\end{equation*}
For another inclusion, it suffices to prove
\begin{equation}\label{eq:VerD0}
{\rm Im}(\Ver_{D_{0},\overline{\mathbf{D}}_{0}}^{\vee})=\{f\in D_{0}^{\vee}\mid f((2\Z/4\Z \times \Z/2\Z)\times \{\bar 0 \})=0\}. 
\end{equation}
Recall that $\Ver_{D_{0},\overline{\mathbf{D}}_{0}}=(\Ver_{D_{0},\overline{D}_{0,i}})_{1\leq i\leq 2}$ and
\begin{equation*}
\Ver_{D_{0},\overline{D}_{0,i}}=\Ver_{D_{0}/D_{0}\cap H_{i},\overline{D}_{0,i}}\circ \pi_{i},
\end{equation*}
where $\pi_{i}\colon D_{0}\rightarrow D_{0}/D_{0}\cap H_{i}$ is the canonical surjection. Since one has
\begin{equation*}
D_{0,1}=D_{0}\cap \widetilde{N}_{1}=2\Z/4\Z \times \{\bar 0 \}\times \Z/2\Z,\quad D_{0}\cap H_{1}=\{0\}, 
\end{equation*}
we obtain that $D_{0}/D_{0}\cap H_{1}$ is not cyclic and $\overline{D}_{0,1}:=D_{0,1}/D_{0}\cap H_{1} \cong \Z/2$. Hence Proposition {3.3} implies $\Ver_{D_{0},\overline{D}_{0,1}}=0$. On the other hand, we have $\Ver_{D_{0},\overline{D}_{0,2}}=\pi_{2}$ by $D_{0,2}:=D_{0}\cap \widetilde{N}_{2}=D_{0}$. Consequently, the homomorphism $\Ver_{D_{0},\overline{\mathbf{D}}_{0}}$ can be written as the composite
\begin{equation*}
D_{0}\xrightarrow{\pi_{2}} D_{0}/D_{0}\cap H_{2}=\overline{D}_{0,2}\xrightarrow{g_2\mapsto (0,g_2)}\overline{D}_{0,1}\times \overline{D}_{0,2}. 
\end{equation*}
Therefore, \eqref{eq:VerD0} follows from the equality $D_{0}\cap H_{2}=(2\Z/4\Z \times \Z/2\Z)\times \{\bar 0 \}$. \qed
\end{proof}

Now we construct pairs of CM fields $K_1,K_2$ satisfying Proposition \ref{prop:neg1}. 

\begin{lem}\label{lm:quartic}
Let $\ell$ be a prime number which is congruent to $1$ modulo $4$, and denote by $K_{(\ell)}$ the unique quartic subfield of $\Q(\zeta_{\ell})$. 
\begin{enumerate}
\item[\emph{(1)}] The field $K_{(\ell)}$ is CM if and only if $\ell \equiv 5\bmod 8$.
\item[\emph{(2)}] We have $\Ram(K_{(\ell)}/\Q)=\{\ell\}$, and $\ell$ is totally ramified in $K_{(\ell)}$. 
\end{enumerate}
\end{lem}

\begin{proof}
(1) Recall that $\Q(\zeta_{\ell})$ is a CM field which is cyclic of degree $\ell-1$ over $\Q$. Since $K_{(\ell)}$ corresponds to the unique subgroup of $\Gal(\Q(\zeta_{\ell})/\Q)$ of order $(\ell-1)/4$, it is CM if and only if $(\ell-1)/4$ is an odd number. Finally, the condition $(\ell-1)/4\notin 2\Z$ is equivalent to the desired congruence $\ell\equiv 5\bmod 8$.

(2) This follows from the equality $\Ram(\Q(\zeta_{\ell})/\Q)=\{\ell\}$ and that $\ell$ is totally ramified in $\Q(\zeta_{\ell})$. \qed
\end{proof}

\begin{lem}\label{lm:prime-tri}
There are infinitely many ordered triples of prime numbers $(\ell_1, \ell_2,\ell_3)$ for which the following are satisfied: 
\begin{enumerate}
\item[\emph{(a)}] $\ell_1\equiv 5\bmod 8$;
\item[\emph{(b)}] $\ell_2 \equiv 1\bmod 4$ and $\left(\frac{\ell_2}{\ell_1}\right)=-1$;
\item[\emph{(c)}] $\ell_3\equiv 3\bmod 4$, $\left(\frac{\ell_3}{\ell_2}\right)=1$ and $\ell_3$ splits completely in $K_{(\ell_1)}$. 
\end{enumerate}
\end{lem}

\begin{proof}
This follows from Dirichlet's prime number theorem. \qed
\end{proof}

For a finite abelian extension $E/\Q$ and a prime number $\ell$, we write $D_{\ell}(E/\Q)$ and $I_{\ell}(E/\Q)$ for the decomposition group and the inertia group of $\Gal(E/\Q)$ at $\ell$ respectively. Observe that if $E'$ is a subextension of $E/\Q$ and let $\pi_{E'}:\Gal(E/\Q) \to \Gal(E'/\Q)$ denote the natural projection, then 
$\pi_{E'}(D_\ell(E/\Q))=D_{\ell}(E'/\Q)$ and $\pi_{E'}(I_\ell(E/\Q))=I_{\ell}(E'/\Q)$.

Theorem \ref{thm:prod-ld} is a consequence of Proposition \ref{prop:K1K2} and Lemma~\ref{prop:KKp}. 

\begin{prop}\label{prop:K1K2}
Let $(\ell_1,\ell_2,\ell_3)$ be an ordered triple of prime numbers as in Lemma \ref{lm:prime-tri}, and set
\begin{equation*}
K_1:=K_{(\ell_1)}(\sqrt{\ell_2}),\quad K_2:=\Q(\sqrt{-\ell_{1}\ell_{3}}). 
\end{equation*}
Then the fields $K_1,K_2$ are CM and they satisfy the conditions {\rm (i)} -- {\rm (iii)} in Proposition \ref{prop:neg1}. 
\end{prop}

\begin{proof}
By definition, $K_2$ is an imaginary quadratic field, and hence CM. Moreover, the condition (a) in Lemma \ref{lm:prime-tri} and Lemma \ref{lm:quartic} (1) imply that $K_1$ is also CM.

In the following, we shall prove that the statements (i), (ii) and (iii) in Proposition \ref{prop:neg1} hold.  Statement (i) follows directly from the definitions of $K_1$ and $K_2$. To prove statement (ii), it suffices to prove the equality $K_1\cap K_2=\Q$. However, this follows from the fact that $\ell_3$ is unramified in $K_1$ and is totally ramified in $K_2$. We now show (iii). It is clear that $L:=K_1K_2$ is unramified outside $2,\ell_1,\ell_2$ and $\ell_3$. Hence it suffices to compute $D_{\ell}(L/\Q)$ for $\ell\in \{2,\ell_1,\ell_2,\ell_3\}$. First, suppose $\ell=2$. Then Lemma \ref{lm:quartic} (2) implies that $2$ is unramified in $K_{(\ell_1)}$. Moreover, by (b) and (c), we have $\ell_1\equiv \ell_2\equiv -\ell_1\ell_3\equiv 1\bmod 4$. Hence $L/\Q$ is unramifed at $2$, which implies that $D_{2}(L/\Q)$ is cyclic. Second, assume $\ell=\ell_1$. By  assumptions (b) and (c), we have the following: 
\begin{gather*}
D_{\ell_1}(K_1/\Q)=G_1,\quad I_{\ell_1}(K_1/\Q)\cong \Z/4\Z \times \{\bar 0\},\\
D_{\ell_1}(K_2/\Q)=I_{\ell_1}(K_2/\Q)=G_2. 
\end{gather*}
Since $\ell_1\neq 2$, by local class field theory $I_{\ell_1}(L/\Q)$ is a finite quotient of $\Z_{\ell_2}^\times$ and is cyclic. Since $I_{\ell_1}(K_2/\Q)=G_2$, we have 
\begin{equation*}
I_{\ell_1}(L/\Q)\cong \langle (\bar 1,\bar 0, \bar 1)\rangle. 
\end{equation*}
The fixed subfield of $I_{\ell_1}(L/\Q)$ in $L$ is $L':=\Q(\sqrt{\ell_2},\sqrt{-\ell_3})$. Indeed, one has 
\[ L^{\<(\bar 2,\bar 0, \bar 0)\>}=\Q(\sqrt{\ell_1},\sqrt{\ell_2},\sqrt{-\ell_1 \ell_3})\quad  \text{ and  }\quad
\Q(\sqrt{\ell_1},\sqrt{\ell_2},\sqrt{-\ell_1 \ell_3})^{\<(\bar 1,\bar 0, \bar 1)\>}=\Q(\sqrt{\ell_2},\sqrt{-\ell_3}). \] 
By (c), we have $D_{\ell_1}(L'/\Q)=\Gal(L'/\Q(\sqrt{-\ell_3}))\simeq \Gal(\Q(\sqrt{\ell_2})/\Q)$, and hence $D_{\ell_1}(L/\Q)/I_{\ell}(L/\Q)$ is generated by the image of $(\bar 0,\bar 1 ,\bar 0 )$ in $G/I_{\ell}(L/\Q)$. Therefore we obtain $D_{\ell_1}(L/\Q)=D_{0}$. 
Third, if $\ell=\ell_2$, then by (b) we have 
\begin{gather*}
D_{\ell_2}(K_1/\Q)=G_1,\quad I_{\ell_2}(K_1/\Q)\cong \{0\}\times \Z/2\Z.
\end{gather*}
One computes $\left(\frac{-\ell_1 \ell_3}{\ell_2}\right)=-1$ and then 
\[ D_{\ell_2}(K_2/\Q)=G_2,\quad I_{\ell_2}(K_2/\Q)=\{0\}. \]
Since $I_{\ell_2}(K_2/\Q)=\{0\}$, one has $I_{\ell_2}(K_2/\Q)=\<(\bar 0,\bar 1, \bar 0)\>$. By $D_{\ell_2}(K_1/\Q)=G_1$, one sees $D_{\ell_2}(L/\Q)=\<(\bar 0,\bar 1, \bar 0), ((\bar 1,\bar 0, \bar c)\>$ for some $\bar c \in \Z/2\Z$. Because $D_{\ell_2}(K_2/\Q)=G_2$, we see $\bar c=\bar 1$ and hence $D_{\ell_2}(L/\Q)=D_0$.
Finally, suppose $i=3$. Since $\ell_3$ splits completely in $K_1$, then 
\begin{equation*}
D_{\ell_3}(L/\Q)=D_{\ell_3}(K_2/\Q)=G_2,
\end{equation*}
and hence it is cyclic. This completes the proof of (iii). \qed
\end{proof}

\begin{lemma}\label{prop:KKp}
Let $(\ell_1,\ell_2,\ell_3)$ and $(\ell'_1,\ell'_2,\ell'_3)$ be ordered triples of prime numbers satisfying (a), (b) and (c) in Lemma \ref{lm:prime-tri}. Put
\begin{equation*}
K:=K_{(\ell_1)}(\sqrt{\ell_2})\times \Q(\sqrt{-\ell_{1}\ell_{3}}),\quad 
K':=K_{(\ell'_1)}(\sqrt{\ell'_2})\times \Q(\sqrt{-\ell'_{1}\ell'_{3}}). 
\end{equation*}
Then we have $K\simeq K'$ if and only if $\ell_i=\ell'_i$ for every $1\leq i\leq 3$. 
\end{lemma}

\begin{proof}
It suffices to prove that $K\simeq K'$ implies $\ell_i=\ell'_i$ for any $1 \leq i\leq 3$. Assume $K\simeq K'$, which is equivalent to $K_{(\ell_1)}(\sqrt{\ell_2})\simeq K_{(\ell'_1)}(\sqrt{\ell'_2})$ and $\Q(\sqrt{-\ell_{1}\ell_{3}})\simeq \Q(\sqrt{-\ell'_{1}\ell'_{3}})$. Then, by Lemma~\ref{lm:quartic} (2), one has
\begin{equation*}
\{\ell_1,\ell_2\}=\Ram(K_{(\ell_1)}(\sqrt{\ell_2}))= \Ram(K_{(\ell'_1)}(\sqrt{\ell'_2}))=\{\ell'_1,\ell'_2\}. 
\end{equation*}
Moreover, the following hold: 
\begin{gather*}
|I_{\ell_1}(K_{(\ell_1)}(\sqrt{\ell_2})/\Q)|=|I_{\ell'_1}(K_{(\ell'_1)}(\sqrt{\ell'_2})/\Q)|=4,\\
|I_{\ell_2}(K_{(\ell_1)}(\sqrt{\ell_2})/\Q)|=|I_{\ell'_2}(K_{(\ell'_1)}(\sqrt{\ell_2})/\Q)|=2. 
\end{gather*}
Hence we must have $\ell_1=\ell'_1$ and $\ell_2=\ell'_2$. On the other hand, $\Q(\sqrt{-\ell_{1}\ell_{3}})\simeq \Q(\sqrt{-\ell'_{1}\ell'_{3}})$ implies $\ell_{1}\ell_{3}=\ell'_{1}\ell'_{3}$ since they are square-free integers. Combining this equality with $\ell_1=\ell'_1$, we obtain $\ell_3=\ell'_3$. This completes the proof. \qed
\end{proof}

\section*{Acknowledgments}
The present project grows up from the 2020 NCTS USRP ``Arithmetic of CM tori''
where the authors participated. 
The authors thank Jeff Achter, Ming-Lun Hsieh, Tetsushi Ito, Teruhisa Koshikawa,
Wen-Wei Li, Thomas R\"ud, Yasuhiro Terakado,
Jiangwei Xue, Seidai
Yasuda for helpful discussions and their valuable comments and especially to Thomas R\"ud who kindly answered the last author's questions and informed his results. They also thank a referee for a helpful comment which improves Proposition~\ref{prop:nKH2Z}.  
Liang, Yang and Yu were partially supported by the MoST grant 109-2115-M-001-002-MY3. Oki was supported by JSPS Research Fellowship for Young Scientists and KAKENHI Grant Number 22J00570.

\appendix

\section{Ono's Conjecture on Tamagawa numbers of algebraic tori \\ Jianing Li and Chia-Fu Yu}

\begin{thm}\label{thm:ono}
For any global field $k$ and any positive rational number $r$, there exists a $k$-torus $T$ such that $\tau(T)=r$.
\end{thm}
\begin{proof}
We follow the idea of Ono. It suffices to construct for each prime $\ell$ (a) a $k$-torus $T_1$ with $\tau(T_1)=\ell$ and (b) a $k$-torus $T_2$ with $\tau(T_2)=\ell^{-2}$. Take a cyclic extension $K/k$ of degree $\ell$ with Galois group $G$ and let $T_1=R_{K/k}^{(1)} \GmK$. Then $H^1(G, X(T_1))\simeq H^2(G,\Z)\simeq \Z/\ell \Z$ and $\Sha^2(K/k,X(T))=0$ by Chebotarev's density theorem. Thus, $\tau(T_1)=\ell$ and (a) is done. 
For (b), we take an abelian extension $K/k$ with Galois group $G\simeq (\Z/\ell\Z)^4$ such that  every decomposition group is cyclic. Take $T_2:=R_{K/k}^{(1)} \GmK$. Since every decomposition group of $G$ is cyclic, $\Sha^2(G,X(T_2))=H^2(G,X(T_2))\simeq H^3(G,\Z)$. By Lyndon's formula \cite[Theorem 6]{lyndon:duke1947}, $H^3(G,\Z)\simeq (\Z/\ell\Z)^6$. On the other hand, $H^1(G,X(T))\simeq H^2(G,\Z)\simeq (\Z/\ell\Z)^4$. Thus, $\tau(T_2)=\ell^{-2}$ and (b) is done. A construction of such an abelian extension $K/k$ when $k$ is a number field is given by Katayama \cite{katayama:1985}. For the function field case, a construction, which is more involved, is given in Theorem~\ref{thm:llll}. \qed  
\end{proof}

\begin{thm}\label{thm:llll}
Let $k$ be a global function field of \char $p>0$. For any prime $\ell$ and any positive integer $n$, there exists an abelian extension of $k$ with Galois group $G\simeq (\Z/\ell\Z)^n$ in which every decomposition group is cyclic. 
\end{thm}

We shall use the cyclotomic function field and we recall the basic facts following from class field theory. For an explicit construction of these fields by Drinfeld modules, we refer to \cite[Chapter 12]{rosen:GTM210}. Let $k=\Fp(t)$ be the function field of the projective line over $\Fp$
 and let $A=\Fp[t]$. In what follows $P, P_i$ always denote monic irreducible polynomials of $A$. Let $\infty$ denote the place of the infinite point and set $V_\infty=\langle t, 1+t^{-1} \Fp[[t^{-1}]] \rangle \subset k^\times_\infty=\Fp((t^{-1}))$. 
 For a monic polynomial $M=\prod_{i=1}^{r} P^{n_i}_i \in A$, let $K(M)$ be the finite abelian extension of $k$ corresponding to 
 the following open subgroup (with finite index) of $\A^\times_k$ 
 \begin{equation*}
U(M) = k^\times  \left (\prod_{i=1}^{r}(1+P^{n_i}_i O_{P_i}) \times V_\infty \times \prod_{v\nmid M\infty} O^\times_v\right).
\end{equation*}
(The field $K(M)$ is called the cyclotomic function field for $M$.) Thus, we have $\A^\times_k/U(M)\simeq \Gal(K(M)/k)$ via the Artin map. Clearly, $P$ is unramified in $K(M)$ if $P\nmid M$. The decomposition group of $\infty$ in $K(M)$ is isomorphic to $\bbF^\times_p$, which is cyclic. There is an isomorphism $(A/M)^\times \cong \A^\times_k/U(M)$ induced by $P\mod M\mapsto (1,\cdots, P,\cdots, 1) \mod U(M)$ ($P$ sitting on the place $P$) for $P\nmid M$, and $a\mod M \mapsto (1,\cdots, 1,\cdots, a) \mod U(M)$ for $a\in \bbF^\times_p \subset (A/M)^\times$ ($a$ sitting on the place $\infty$). So we have the following isomorphism which maps $P \mod M$ ($P\nmid M$) to its Frobenius element in $\Gal(K(M)/k)$
\begin{equation*}
(A/M)^\times \simeq \Gal(K(M)/k).
\end{equation*}

 We will primarily be concerned with the fields $K(P)$ and $K(P^2)$, in which $P$ is totally ramified.
 If $\ell$ divides $|(A/P)^\times|$, let $F(P)$ denote the subfield of $K(P)$ fixed by $((A/P)^\times)^\ell$ so that $\Gal(F(P)/k)\simeq \Z/\ell\Z$. If $a\in A/P$, 
then $1+aP \bmod P^2 \in (A/P^2)^\times$ is well defined and it generates the subgroup $H(a)=\{1+aiP\bmod P^2: i=0,1,\cdots,p-1\}$ of order $p$ when $a\neq 0$. We let $F(P,a)$ denote the subfield of $K(P^2)$ fixed by $\{x\in (A/P^2)^\times: x^{q-1}\in H(a)\}$ where $q=p^{\deg P}=|A/P|$. This group is isomorphic to $H(a) \times (A/P)^\times$; since $(A/P^2)^\times \simeq A/P\times (A/P)^\times$, we have $\Gal(F(P,a)/k) \simeq (\Z/p\Z)^{\deg P-1}$ when $a\neq 0$.

\begin{proof}
We first prove the theorem when $k=\Fp(t)$ and $\ell\neq p$. The argument of this case is entirely similar to the case of number fields given in \cite{katayama:1985}.
Choose $P_1$ such that it splits in $k(\mu_{\ell})$ where $\mu_{\ell}$ is the group of $\ell$-th roots of unity. We inductively choose $P_r$ ($1\leq r \leq n$) such that $P_r$ splits completely in the composite field $k(\mu_{\ell}, \sqrt[\ell]{P_1},\cdots, \sqrt[\ell]{P_{r-1}})F(P_1)\cdots F(P_{r-1})$. Clearly, each $P_i$ has infinitely many ways to choose by density theorems. Now let $K$ be the composite field $F(P_1)\cdots F(P_n)$. The decomposition group of $\infty$ in $K/k$ is cyclic, since $K$ is a subfield of the cyclotomic function field $K(P_1\cdots P_r)$. Moreover, if $P$ is not one of the $P_i$, $P$ is unramified and hence its decomposition group is cyclic. Therefore, to show that $K/k$ is the desired extension, it suffices to show that $P_i$ splits in $F(P_j)$ whenever $i\neq j$. Assume $i <j$. Then $P_j$ splits in $F(P_i)$ by the construction. Conversely, since $P_j$ splits in $k(\sqrt[\ell]{P_i})$, $P_i$ is an $\ell$-power in the completion $k_{P_j}$ which implies that $P_i \in ((A/P_j)^\times)^\ell$; hence $P_i$ splits in $F(P_j)$. This proves the theorem when $k=\Fp(t)$ and $\ell \neq p$.

Assume next $k=\Fp(t)$ and $\ell = p$. Take an arbitrary $P_1\in A$ with $a_1=0\in A/P_1$. Let $k_1$ be a subfield of $F(P_1,a_1)$ with $[k_1:k]=p$. We inductively choose $(P_r,a_r,k_r)$ $(1\leq r\leq n)$ with $a_r\in A/P_r$ such that $\deg P_r\geq r$ and
\begin{equation*}
P^{q_i-1}_r \equiv 1+a_iP_i  \mod P^2_i \quad  \text{ for } i=1,\cdots, r-1, \text{ where } q_i = |A/P_i|. 
\end{equation*}
Let $a_{ri}\in A/P_r$ ($i=1,\cdots, r-1$) such that 
\begin{equation*}
P^{q_r-1}_{i} \equiv 1+a_{ri}P_r \mod P^2_r \quad \text{ for } i=1,\cdots, r-1.
\end{equation*}
Let $a_r:=a_{r1}$, and let $k_r$ be a subfield of $\cap{_{i=1}^{r-1}} F(P_r, a_{ri})$ with $[k_r:k]=p$. Put $K:=k_1\cdots k_r$ and let us show that $K/k$ is the desired extension. We have $\Gal(K/k)\simeq (\Z/p\Z)^n$ by considering the ramification. Assume $i< j$. Then $P_j$ splits completely in $F(P_i, a_i)$ by the first congruences and hence splits in $k_i$; conversely, $P_i$ splits completely in $F(P_j, a_{ji})$ by the second congruences and hence also splits in its subfield $k_j$. It follows that the decomposition subgroup of $P_i$ ($i=1,\cdots, n$) in $K/k$ is cyclic. This property also holds for the place $\infty$, since $K\subset K(P^2_1\cdots P^2_n)$. This proves the theorem when $k=\Fp(t)$ and $\ell = p$. 

Now assume that $k$ is any finite extension of $\Fp(t)$ and $\ell$ is any prime. Choose $P_1, \cdots, P_n$ as above but we additionally require that $P_i$ is unramified in $k/\Fp(t)$, and we define $K/\Fp(t)$ as above. Then $\Gal(Kk/k) \simeq (\Z/\ell\Z)^n$. Since each decomposition subgroup for $Kk/k$ is a subgroup of that of $K/\Fp(t)$, the field $Kk$ is the desired extension. This completes the proof of Theorem~\ref{thm:llll}. \qed
\end{proof}



\bibliographystyle{plain}
\bibliography{TeXBiB}

\end{document}